\documentclass[a4paper,twoside,12pt]{article}
\usepackage{amsmath,amsthm, amssymb,color}
\usepackage{multicol}
\usepackage{wrapfig}
\usepackage{graphicx}
\usepackage{graphics}
\usepackage{yfonts}
\usepackage{times}
\usepackage{float}
\usepackage{enumerate}

\usepackage{chngcntr}

\usepackage{mathrsfs}

\theoremstyle{definition}
\newtheorem{thm}{Theorem}[section]
\newtheorem{cor}[thm]{Corollary}
\newtheorem{lem}[thm]{Lemma}

\newtheorem{rem}[thm]{Remark}
\newtheorem{exa}[thm]{Example}

\newcommand{\subjclass}[1]{\bigskip\noindent\emph{2010 Mathematics Subject Classification: }\enspace#1}
\newcommand{\keywords}[1]{\noindent\emph{Keywords:}\enspace#1}

\numberwithin{equation}{section}
\def\mean#1{\left\{\!\left|\,#1\,\right|\!\right\}}
\begin{document}

%%%%% To ease editing, add:

\baselineskip=17pt

%%%%%%%%%%%%%%%%

\title{Thin obstacle problem: \\ estimates of the distance to the exact solution}

\author{{\textsc{ Darya E.  Apushkinskaya}}\\
Department of Mathematics, Saarland University, P.O. Box 151150 \\
66041 Saarbr{\"u}cken, Germany\\
\textit{E-mail: darya@math.uni-sb.de}\\
Peoples’ Friendship University of Russia (RUDN University)\\
6 Miklukho-Maklaya St, Moscow, 117198, Russian Federation \\
\\
{\textsc{ Sergey I. Repin}}\\
Steklov Institute of Mathematics at St. Petersburg,\\ Fontanka 27, 191023 St. Petersburg,
Russian Federation\\
\textit{E-mail: repin@pdmi.ras.ru}\\
University of Jyv{\"a}skyl{\"a}, P.O. Box 35  (Agora),\\
FIN-40014, Finland \\
\textit{E-mail:serepin@jyu.fi}}

\date{}

\maketitle
\bibliographystyle{alpha}
\newcommand{\tg}{\textup{tg}}
\newcommand{\ep}{\varepsilon}
\newcommand{\osc}[1]{\underset{#1}{\textup{osc}}}

\begin{abstract}
We consider elliptic variational inequalities generated by
obstacle type problems with thin obstacles. For this class of problems,
we deduce estimates of the distance (measured in terms of the natural
energy norm) between the exact solution and any function that satisfies the
boundary condition and is admissible with respect to the obstacle condition
(i.e., it is valid for any approximation regardless of the method by which
it was found).
Computation of the estimates does not require knowledge of the exact solution
and uses only the problem data and an approximation. The estimates
provide  guaranteed upper bounds of the error (error majorants) and
vanish if and only if the approximation coincides with the exact solution.
In the last section, the efficiency of error majorants is confirmed by an example, where the exact solution is known.


\end{abstract}
\section{Introduction}

Let $\Omega$ be an open, connected, and bounded domain in $\mathbb{R}^n$ with Lipschitz continuous boundary $\partial\Omega$,
and let $\mathcal{M}$ be a smooth $(n-1)$-dimensional manifold in $\mathbb{R}^n$,  which divides $\Omega$ into two Lipschitz subdomains $\Omega_+$ and $\Omega_-$. 
Throughout the paper, we use the standard notation for the Lebesgue
and Sobolev spaces of functions.
Since no confusion may arise, we denote  the norm in $L^2\left( \Omega \right) $ and  the norm in the space $L^2\left( \Omega, \mathbb{R}^n\right) $ 
containing vector valued functions by one common symbol  $\|\cdot\|_{\Omega}$.

For given functions $\psi: \mathcal{M}\rightarrow \mathbb{R}$ and $\varphi : \partial\Omega\rightarrow \mathbb{R}$ satisfying $\varphi \geq \psi$ on $\mathcal{M} \cap \partial\Omega$, we consider the following variational Problem
($\mathcal P$): minimize the functional
\begin{equation}
\label{Dint}
J(v)=\frac{1}{2}\int\limits_{\Omega}|\nabla v|^2dx
\end{equation}
over the closed convex set
$$
\mathbb{K}=\left\lbrace v \in H^{1}\left( \Omega\right) : \quad v \geqslant \psi\ \text{on}\ \mathcal{M}\cap \Omega, \quad v = \varphi\ \text{on}\ \partial\Omega\right\rbrace .
$$
Here, $\varphi \in H^{1/2}(\partial\Omega)$ and the function $\psi$ is supposed to be smooth.

\vskip-0.5cm
\begin{figure}[htbp]
\centering
\includegraphics[width=0.8\textwidth]{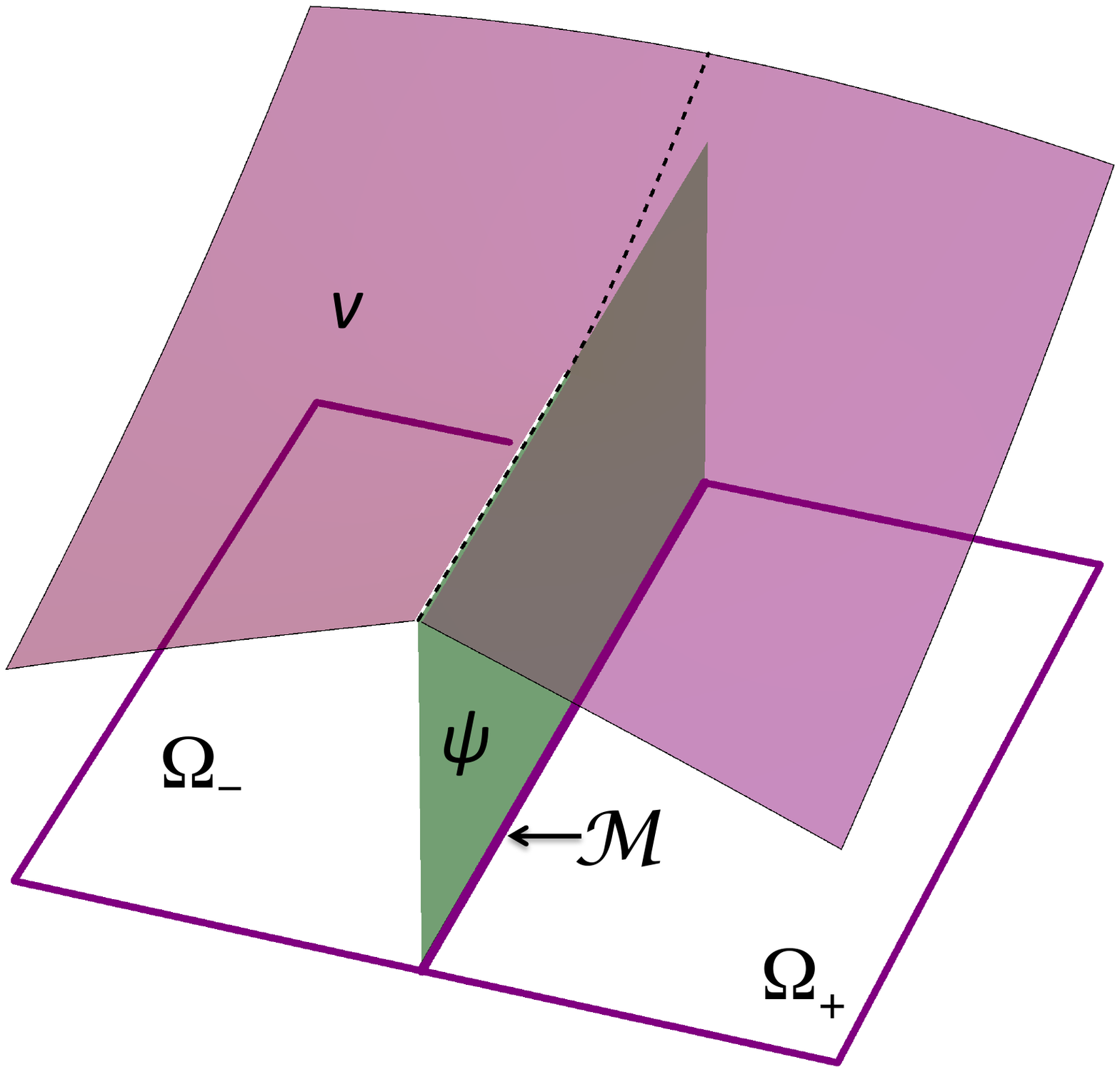}
\caption{The thin obstacle problem}
\label{bild-thin-obstacle}
\end{figure}

Problem ($\mathcal{P}$) is called the \textit{thin obstacle problem} associated with the
 thin obstacle $\psi$. In many respects, it  differs from the classical obstacle 
problem where the constrain $v\geq\psi$ is imposed on the entire domain $\Omega$.
This mathematical model arises in various real life
problems.
In the  $2D$ case 
(see Fig.~\ref{bild-thin-obstacle}),\, it describes equilibrium of an elastic membrane  above a very thin object (e.g., see \cite{KO88}). The well known Signorini problem belongs to the same class
of mathematical models. Similar models appear in continuum mechanics, e.g., in  temperature control problems
and in analysis of flow through semi-permeable walls subject to the phenomenon of osmosis (see, e.g., \cite{DLi76}). Thin obstacle problems also arise in financial mathematics if the random variation of an underlying asset changes discontinuously (see \cite{CT04},\cite{Si07},  \cite{PSU12} and the references therein).
\vspace{0.2cm}

The problem ($\mathcal{P}$) is an example of  a variational inequality,
which mathematical analysis goes back to the fundamental paper \cite{LS67}. 
Existence of the unique minimizer $u \in \mathbb{K}$   is well known (see  \cite{LS67} and also the  books 
\cite{Rod87}, \cite{Fr88} and \cite{KS80}).    For smooth $\mathcal{M}$ and $\psi$ it is also known that  $u \in C^{1, \alpha}_{loc}\left( \Omega_{\pm} \cup \mathcal{M}\right) $ with $0<\alpha \leqslant 1/2$ (see \cite{Caf79}, \cite{Ura85}, \cite{AC04} and the book \cite{PSU12}). 
This optimal regularity of $u$ guarantees that $\frac{\partial u\ }{\partial \textbf{n}_+}$ and $\frac{\partial u\ }{\partial \textbf{n}_-}$ belong to $L^2(\mathcal{M})$, where $\textbf{n}_{\pm}$ denote the outer unit normals to $\Omega_{\pm}$ on $\mathcal{M}$. 
It is also easy to see that
the minimizer $u$ satisfies the harmonic equation
$
\Delta u =0$ in the subdomains $\Omega_+$ and  $\Omega_-$,
but in general $u$ is not a harmonic function in $\Omega$. Instead, on $\mathcal{M}$, we have the so-called \textit{complimentarity conditions}
\begin{equation}
\label{complementarycond}
u-\psi \geqslant 0,\qquad
\left[\frac{\partial u}{\partial{\mathbf{n}}}\right]\geqslant 0,\qquad (u-\psi)\left[\frac{\partial u}{\partial{\mathbf{n}}}\right] =0,
\end{equation}
where $\left[\frac{\partial u}{\partial{\mathbf{n}}}\right]:=\frac{\partial u\ }{\partial \textbf{n}_+}+\frac{\partial u\ }{\partial \textbf{n}_-}$ is the jump of $\nabla u\cdot \mathbf{n}$ across $\mathcal{M}$. 
Here and later on $\cdot$ denotes the inner product in $\mathbb{R}^n$.
\vspace{0.2cm}

Thin obstacle problems have been actively studied
from the early 1970s. These studies
were mainly focused  either  on
regularity of  minimizers  (see \cite{Fre75}, \cite{Fre77}, \cite{Ric78}, \cite{Caf79}, \cite{Ura85}, \cite{AC04}, \cite{Gui09}) or  on properties of the
respective  free boundaries (see \cite{Lew72}, \cite{ACS08}, \cite{CSS08}, \cite{GP09}, \cite{KPS15}, and \cite{SiSa16}). A systematic 
overview of these results can be found in the book \cite{PSU12}.
\vspace{0.2cm}

In this paper, we are concerned with a different question. Our analysis is focused not on properties of  the exact minimizer,  but on estimates of the distance (measured in terms of the natural energy norm) between $u$  and any function $v\in \mathbb{K}$. In other words,  we wish to obtain estimates able to
detect which neighborhood of $u$ contains a function $v$ (considered as an approximation of the minimizer). 

These estimates are fully computable, i.e., they depend only on $v$ (which is
assumed to be known) and on the data of the problem ( the exact solution $u$  and the respective exact coincident set $\{u=\psi\}$ do not enter the estimate explicitly). 
A general approach to the derivation of such type estimates based
on methods of the duality theory in the calculus of variations
is presented in \cite{Re00a}.
%It is based on using of duality method in the calculus of variations, which is useful in various variational and optimization problems (see, for instance, \cite{DLi76}, \cite{EkTe76}, \cite{IT79}, \cite{BoSh00}).}
For the classical obstacle problem (which solution is bounded in $\Omega$ from above and below by two obstacles) analogous estimates 
were obtained in (\cite{Re03}).  For the two-phase obstacle problem
(which was introduced in \cite{W01} and  studied from regularity point ov view in \cite{Ura01}, \cite{SUW04}, \cite{SW06}, and \cite{SUW07}) similar
estimates has been recently derived in \cite{RV15}. These results were
obtained by methods of the duality theory in the calculus of variations,
which are widely used for analysis of various variational and optimization
problems (e.g., see 
 \cite{DLi76}, \cite{EkTe76}, \cite{IT79}, \cite{KS80}, \cite{BoSh00}). 

It should be noted that getting explicit estimates of errors is based upon
the general relations exposed in \cite{Re00a}, is not at all a straightforwarding and simple matter. In this context, there is a clear difference with   the results mentioned above. Indeed, the estimate (\ref{27-a})  contains an integral term related to the lower dimensional set $\mathcal{M}$. Therefore,   our analysis will require   estimates with explicitly known constants  for the traces of functions  on   $\mathcal{M}$. For this purpose,  we will  introduce and analyze an auxiliary variational problem, which generates constants in special Poincar{\'{e}}-type inequalities valid for functions with zero mean boundary traces.
\vspace{0.2cm}

The main results are presented in Theorems
\ref{simple_majorant}, \ref{main_theorem},
and  \ref{theorem_Poincare},  that 
suggest different  majorants  of the norm
$\|\nabla (v-u)\|_{\Omega}$.
The majorants are nonnegative and vanish if and only if $v$ coincides with $u$.
%In Sec.~4, we discuss optimisation of the estimates with respect to certain parameters entering  the majorants. 
Section~4  is devoted to the \textit{boundary thin obstacle problem} (also known as the \textit{scalar Signorini problem}). 
Finally, in the last section we consider an example,
where the exact solution of a thin obstacle problem is known. 
We find the exact distance between this solution and some selected functions
$v$ and show that our estimates provide correct upper bounds of the distance.

\section{Estimates of the distance to the exact solution}
Let $u\in \mathbb{K}$ be a minimizer of  variational problem ($\mathcal{P}$). Elementary calculations
yield the identity 
$$
J(v)-J(u)=\frac{1}{2}\|\nabla (v-u)\|_{\Omega}^2-\|\nabla u\|_{\Omega}^2+\int\limits_{\Omega}\nabla v \cdot \nabla udx,
$$
which holds for every $v\in \mathbb{K}$.
Since $u$ satisfies the respective variational inequality,
%$(\nabla u, \nabla (v-u))\geq 0$
we conclude that
\begin{equation}
\frac{1}{2}\|\nabla \left( v-u\right) \|_{\Omega}^2 \leqslant J(v)-J(u), \qquad \forall v\in \mathbb{K}.
\label{basic-deviation}
\end{equation}
The inequality (\ref{basic-deviation}) does not provide a computable majorant of the distance between $u$ and $v$ because the value $J(u)$ is unknown.
Therefore, our goal is to replace  the difference $J(v)-J(u)$ in (\ref{basic-deviation}) by a fully computable quantity. %independing on $u$. 
%For this purpose, we modify  a method  suggested in \cite{Re03} for the classsical obstacle problem.

%will follow the concepts of perturbed and dual problems presented in \cite{Re03}.
%%%%%%%%%%%%%%%%%%%%%%%%%%%%%%%%%%%%%%%%%%%%%%%%%%%%%%%%%%%%%%%%%%%
\subsection{The first form of the  majorant}
For any $\lambda \in \Lambda:=\{
\lambda \in L^2(\mathcal{M})\,:\,
\lambda (x) \geqslant 0\, \text{a.e. on}\,\mathcal{M}\}$, we introduce the {\it perturbed} functional
$$
J_{\lambda}(v):=J(v)-\int\limits_{\mathcal{M}}\lambda  \left( v-\psi\right) 
d\mu.
$$
It is easy to see that
$$
\sup\limits_{\lambda \in \Lambda}J_{\lambda}(v)=
J(v)-\inf\limits_{\lambda \in \Lambda}\int\limits_{\mathcal{M}}\lambda (v-\psi) d\mu =
\left\{
\begin{aligned}
&J(v), \quad  \text{if} \ v\geqslant \psi\ \text{on}\ \mathcal{M},\\
&+\infty, \quad \text{otherwise.}
\end{aligned} \right. \quad ,
$$
Hence,
\begin{equation}
J(u)=\inf\limits_{v\in \mathbb{K}}J(v)=\inf\limits_{v\in \varphi+H^1_0(\Omega)}\sup\limits_{\lambda \in \Lambda} J_{\lambda}(v),
\label{J-via-J-lambda}
\end{equation}
where $\varphi+H^1_0(\Omega)
:=\left\{w=\varphi+v : v\in H^1_0\left(\Omega\right)\right\}$, and $H^1_0(\Omega)$ is a subspace of $H^1(\Omega)$ containing the functions
vanishing on the boundary.

The functional $J_\lambda$ generates  the following  variational problem $(\mathcal{P}_{\lambda})$:
 find $u_{\lambda} \in \varphi+H^1_0(\Omega)$ such that
\begin{equation}
\label{Plambda}
J_{\lambda}(u_{\lambda}):=\inf\limits_{v\in \varphi+H^1_0(\Omega)}J_{\lambda}(v).
\end{equation}
Since $\varphi + H^1_0(\Omega)$ is the affine subspace of $H^1(\Omega)$ and $J_{\lambda}$ is a quadratic functional,  the results of \cite{LS67} imply   unique solvability of the problem ($\mathcal{P}_{\lambda}$) for any  $\lambda \in \Lambda$. Moreover, in view of (\ref{J-via-J-lambda}), 
 $J(u)$ is bounded from below by the quantity $J_{\lambda}(u_{\lambda})$. Indeed,  
\begin{equation}
J(u)=\inf\limits_{v \in \varphi +H^1_0(\Omega)}\sup\limits_{\lambda \in \Lambda}J_{\lambda}(v)
\geqslant \sup\limits_{\lambda \in \Lambda}\inf\limits_{v \in \varphi +H^1_0(\Omega)}J_{\lambda}(v) \geqslant J_{\lambda}(u_{\lambda}) \quad \forall \lambda \in \Lambda.
\label{estimate-J-below}
\end{equation}
%where the first inequality follows from Lemma~5.2.1 \cite{NeRe04}. \vspace{0.2cm}

%To find a computable lower bound of $J_{\lambda}(u_{\lambda})$, we use  
The dual counterpart of  ($\mathcal{P}_{\lambda}$)
is generated by  the Lagrangian
$$
\mathcal{L}_{\lambda}(v, y^*):=\int\limits_{\Omega}\left( y^*\cdot \nabla v-\frac{1}{2}|y^*|^2\right) dx-\int\limits_{\mathcal{M}}\lambda  (v-\psi) d\mu,
$$
which is defined on the set $( \varphi +H^1_0(\Omega)) \times L^2\left( \Omega, \mathbb{R}^n\right)$.
%where $\mathbb{Y}^*:=L^2\left( \Omega, \mathbb{R}^n\right)$. 
Obviously,
$$
J_{\lambda}(v)=\sup\limits_{y^* \in {L^2\left( \Omega, \mathbb{R}^n\right)}}
 \mathcal{L}_{\lambda}(v, y^*)
$$
and the corresponding  dual  functional $J^{*}_{\lambda}$
is defined  by the relation
$$
J^*_{\lambda}(y^*):=\inf\limits_{v\in \varphi +H^1_0(\Omega)}\mathcal{L}_{\lambda}(v,y^*).
$$
It is not difficult to see that
$$
J^{*}_{\lambda}(y^*):=\left\lbrace 
\begin{aligned}
\int\limits_{\Omega}\left( y^* \cdot \nabla \varphi -\frac{1}{2}|y^*|^2\right)dx -\int\limits_{\mathcal{M}}\lambda \left(\varphi-\psi \right) d\mu  &\quad \text{if}\quad y^*\in Q^{*}_{\lambda, \mathcal{M}},\\ 
\qquad \qquad -\infty \qquad \qquad \qquad \qquad &\quad \text{if}\quad y^* \notin Q^{*}_{\lambda, \mathcal{M}},
\end{aligned} \right.
$$
where 
$$
Q^{*}_{\lambda, \mathcal{M}}:=\left\{ y^* \in {L^2\left( \Omega, \mathbb{R}^n\right)} : \int\limits_{\Omega}y^* \cdot \nabla w dx=\int\limits_{\mathcal{M}}\lambda w  d\mu  \quad \forall w\in H^1_0\left( \Omega\right) \right\} .
$$
The set $Q^{*}_{\lambda, \mathcal{M}}$ contains functions that
satisfy (in the generalized sense) the equation ${\rm div} y^*=0$ in $\Omega_-$ and $\Omega_+$ and the
condition $[y^*\cdot {\bf n}]=\lambda$ on $\mathcal M$ (here
$[y^*\cdot  {\bf n}]$ denotes the jump of $y^*\cdot  {\bf n}$) .
The functional $J^{*}_{\lambda}$ generates a new variational Problem
($\mathcal{P}^*_{\lambda}$) (dual to ($\mathcal{P}_{\lambda}$)): find  $y_{\lambda}^* \in Q^{*}_{\lambda, \mathcal{M}}$ such that
$$
J^{*}_{\lambda}(y^{*}_{\lambda}):=\sup\limits_{y^*\in Q_{\lambda, \mathcal{M}}} J^{*}_{\lambda}(y^*).
%\eqno{(\mathcal{P}^{*}_{\lambda})}f
$$
This is a quadratic maximization problem with a strictly concave and continuous functional. Well known results of convex analysis (see, e.g., \cite{EkTe76})
guarantee that it has a unique maximizer in the affine subspace $Q^{*}_{\lambda, \mathcal{M}}$. 
Moreover, we have\, the duality relation
\begin{equation}
J_{\lambda}(u_{\lambda})=\inf\limits_{v\in \varphi+H^1_0(\Omega)}J_{\lambda}(v)=\sup\limits_{y^*\in Q^{*}_{\lambda, \mathcal{M}}}J^{*}_{\lambda}(y^{*})=J^{*}_{\lambda}(y^*_{\lambda}).
\label{duality}
\end{equation}
Combining (\ref{estimate-J-below}) and (\ref{duality}), we deduce the estimate
\begin{align*}
J(v)-J(u) &\leqslant J(v)-J^{*}_{\lambda}(y^{*}_{\lambda})=J(v)-\sup\limits_{y^*\in Q^{*}_{\lambda, \mathcal{M}}}J^*_{\lambda}(y^*)\\
&=J(v)+\inf\limits_{y^*\in Q^{*}_{\lambda, \mathcal{M}}}\left( -J^*_{\lambda}(y^*)\right) =
\inf\limits_{y^*\in Q^{*}_{\lambda, \mathcal{M}}}\left[ J(v)-J^*_{\lambda}(y^*)\right] .
\end{align*}
Therefore, the inequality
\begin{equation}
J(v)-J(u) \leqslant J(v) -J^*_{\lambda}(y^*)
\label{analog-26-ReVal}
\end{equation}
holds true for all $v\in \mathbb{K}$, all $\lambda \in \Lambda$, and all $y^*\in Q^{*}_{\lambda, \mathcal{M}}$.

Thanks to the assumption $\varphi \in H^{1/2}(\partial\Omega)$, the boundary datum $\varphi$ allows a continuation as $H^1$-function on the whole set $\Omega$. We will preserve the notation $\varphi$ for the extended function.
Since $y^* \in Q^{*}_{\lambda, \mathcal{M}}$ and $v-\varphi\in H^1_0(\Omega)$
for any $v\in \mathbb{K}$, we find that
\begin{equation*}
\int\limits_{\Omega}y^{*} \cdot \nabla\varphi dx =\int\limits_{\Omega}y^{*}\cdot \nabla v dx-\int\limits_{\Omega}y^{*}\cdot \nabla (v-\varphi)dx=
\int\limits_{\Omega}y^{*}\cdot \nabla vdx -\int\limits_{\mathcal{M}}\lambda (v-\varphi)  d\mu .
\end{equation*}
Now the right-hand side of (\ref{analog-26-ReVal}) can be rewritten as follows:
\begin{multline}
\label{analog-27-ReVal}
J(v)-J^*_{\lambda}(y^*)=\int\limits_{\Omega}\left(\frac{1}{2}|\nabla v|^2+\frac{1}{2}| y^{*}|^2-y^{*}\cdot \nabla\varphi \right)dx+\int\limits_{\mathcal{M}}\lambda (\varphi -\psi) d\mu \\
=\frac{1}{2}
\int\limits_{\Omega}|\nabla v-y^*|^2dx +\int\limits_{\mathcal{M}}\lambda \left( v-\psi\right)  d\mu .
\end{multline}
Combination of (\ref{basic-deviation}), (\ref{analog-26-ReVal}) and (\ref{analog-27-ReVal}) yields the following upper bound of the error:

\begin{thm} \label{simple_majorant}
For any $v\in \mathbb{K}$, the distance to the minimizer $u$
is subject to the estimate
\begin{equation}
\label{27-a}
\|\nabla (v-u)\|_{\Omega}^{2}\leqslant \|\nabla v-y^*\|^2_{\Omega} +2\int\limits_{\mathcal{M}}\lambda \left( v-\psi\right)  d\mu ,
\end{equation}
where  $\lambda$ and $y^*$ are arbitrary functions in $\Lambda$ and $Q^*_{\lambda, \mathcal{M}}$, respectively.
\end{thm}

Theorem~\ref{simple_majorant} can be viewed as a generalized form of the hypercircle estimate (see \cite{PrS47} and \cite{M64})  for the considered class of problems.

\begin{rem}
\label{rem2}
Define the coincidence sets associated with  $u$ and  $v$:
$$
{\mathcal M}^u_\psi:=\{x\in {\mathcal M}\,:\,u(x)=\psi(x)\,\}\qquad{\rm and}\qquad 
{\mathcal M}^v_\psi:=\{x\in {\mathcal M}\,:\,v(x)=\psi(x)\,\}.
$$
Assume that ${\mathcal M}^u_\psi\subset {\mathcal M}^v_\psi$. In this case, the estimate (\ref{27-a}) is sharp in the sense that there exist $y^*$
and $\lambda$ such that the inequality holds as the equality.
Indeed, let $y^*=p^*:=\nabla u$ and $\lambda_*=\left[p^*\cdot {\bf n}\right]$.
Evidently, $p^*\in Q^*_{\lambda_*, \mathcal{M}}$. In view of
(\ref{complementarycond}), 
$\lambda_*=0$ on $\mathcal{M}\setminus {\mathcal M}^u_\psi$. Since
$\mathcal{M}\setminus {\mathcal M}^v_\psi\subset \mathcal{M}\setminus {\mathcal M}^u_\psi$, we conclude that
$$
\int\limits_{\mathcal{M}}\lambda_* \left( v-\psi\right)  d\mu=
\int\limits_{\mathcal{M}\setminus {\mathcal M}^v_\psi}\lambda_* \left( v-\psi\right)  d\mu=0.
$$
Hence, the right hand side of (\ref{27-a}) coincides with the left one.
\end{rem}

%%%%%%%%%%%%%%%%%%%%%%%%%%%%%%%%%%

\subsection{Advanced forms of the majorant}
Inequality (\ref{27-a}) provides a simple and transparent form of the upper bound, but it operates with the  set
$Q^*_{\lambda, \mathcal{M}}$, which is defined by means of differential type conditions. This set is rather narrow and inconvenient if we wish to use simple approximations. In this section, we
overcome this drawback and
%requires that $y^{*}$ belongs to the complicated set $Q^*_{\lambda, \mathcal{M}}$. 
%The right-hand side of (\ref{27-a}) is fully computable, but it requires that $y^{*}$ satisfies the restriction $y^*\in Q^{*}_{\lambda, \mathcal{M}}$. 
%By this reason, we will 
replace (\ref{27-a}) by a more general estimate valid for functions in the set
$$
 H\left( \Omega_{\pm}, \text{\rm{div}}\right)
:=\left\{q^*\in L^2 \left(\Omega, \mathbb{R}^n \right) : \text{\rm{div}}\,(q^*|_{\Omega_{\pm}} )\in L^2\left(\Omega_{\pm}\right),\; \left[q^*\cdot \mathbf{n}\right]\in L^2({\mathcal M})\right\},
$$ 
which is much wider than $Q^*_{\lambda, \mathcal{M}}$.

\begin{lem} \label{lemma2.2}
Let $q^* \in H(\Omega_{\pm}, \text{\rm{div}})$, and let $\lambda \in \Lambda$. 
Then
\begin{equation}
\begin{aligned}
\inf\limits_{y^* \in Q^*_{\lambda, \mathcal{M}}}\|q^*-y^*\|_{\Omega}
&\leqslant C_{F_{\Omega_+}} \|\text{\rm{div}}\, q^*\|_{\Omega_+}+C_{F_{\Omega_-}} \|\text{\rm{div}}\, q^*\|_{\Omega_-}\\
&+C_{Tr_{\mathcal{M}}}\|\lambda -[q^*\cdot \mathbf{n}]\|_{\mathcal{M}},
\end{aligned}
\label{replace_q}
\end{equation}
where  $C_{Tr_{\mathcal{M}}}$ and $C_{F_{\Omega_{\pm}}}$  are the constants defined by  (\ref{Def-trace-minimum}) and (\ref{Fried-ineq}), respectively.
\end{lem}
\begin{proof} 
Consider an auxiliary variational problem ($\mathcal{P}_{q^*})$: minimize the functional
$$
\mathcal{J}_{q^*}(w)=\int\limits_{\Omega}\left(\frac{1}{2}|\nabla w|^2+q^*\cdot \nabla w\right)dx -\int\limits_{\mathcal{M}}\lambda w  d\mu
$$
on the space $H^1_0(\Omega)$.
For any given $q^* \in H(\Omega_{\pm}, \text{\rm{div}})$ and $\lambda \in \Lambda$, the functional $\mathcal{J}_{q^*}$ is convex, continuous, and coercive on $H^1_0(\Omega)$. Hence the problem $\mathcal{P}_{q^*}$
has a unique minimizer $w_{\lambda,q^*}\in H^1_0$ . 
\vspace{0.1cm}

Since $q^*\in H\left( \Omega_{\pm}, \text{\rm{div}}\right)$, the functional 
$\mathcal{J}_{q^*}$ has the form
\begin{equation}
\label{equa_form}
\begin{aligned}
\mathcal{J}_{q^*}(w)&=\int\limits_{\Omega}\frac{1}{2}|\nabla w|^2dx-\int\limits_{\Omega_+}w\text{\rm{div}}\, q^*dx-\int\limits_{\Omega_-}w\text{\rm{div}}\, q^*dx \\
&-\int\limits_{\mathcal{M}}\left(\lambda - [q^* \cdot \mathbf{n}]\right) w d\mu .
\end{aligned}
\end{equation}
For any $\widetilde{w}\in H^1_0(\Omega)$, the minimizer 
$w_{\lambda,q^*}$ 
satisfies the identity
\begin{equation}
\begin{aligned}
\int\limits_{\Omega}\nabla w_{\lambda,q^*}\cdot \nabla \widetilde{w}dx
%\int\limits_{\Omega_-}\nabla w_{\lambda,q^*}\cdot \nabla \widetilde{w}dx
=\int\limits_{\Omega_+} \widetilde{w}\text{\rm{div}}\,q^* dx+\int\limits_{\Omega_-} \widetilde{w}\text{\rm{div}}\,q^* dx
+\int\limits_{\mathcal{M}}\left(\lambda- \left[q^* \cdot \mathbf{n} \right]  \right) \widetilde{w}  d\mu .
\end{aligned}
\label{identity1}
\end{equation}
We set $\widetilde w=w_{\lambda,q^*}$ and use the estimate
\begin{equation}
\begin{aligned}
\int\limits_{\mathcal{M}}\left(\lambda- \left[q^* \cdot \mathbf{n} \right]  \right) w_{\lambda,q^*}  d\mu  &\leqslant
C_{Tr_{\mathcal{M}}}(\Omega_{\pm} )\|\nabla w_{\lambda,q^*}\|_{\Omega_{\pm}}\,\|\lambda-\left[q^* \cdot \mathbf{n} \right]\|_{\mathcal{M}}\\
&\leqslant C_{Tr_{\mathcal{M}}}\|\nabla w_{\lambda,q^*}\|_{\Omega}\,\|\lambda-\left[q^* \cdot \mathbf{n} \right]\|_{\mathcal{M}},
\end{aligned}
\label{estimate_on_M}
\end{equation}
where 
\begin{equation}
C_{Tr_{\mathcal{M}}}:=\min\{C_{Tr_{\mathcal{M}}}(\Omega_{+}), C_{Tr_{\mathcal{M}}}(\Omega_{-})\}
\label{Def-trace-minimum}
\end{equation}
and the constants come from the trace inequalities
$$
\|w\|_{\mathcal{M}} \leqslant C_{Tr_{\mathcal{M}}}(\Omega_{\pm}) \|\nabla w\|_{\Omega_{\pm}}.
$$
Two other terms in the right hand side of (\ref{identity1})
are estimated by the Friedrich's  type inequalities 
\begin{equation} \label{Fried-ineq}
\|w\|_{\Omega_{\pm}} \leqslant C_{F_{\Omega_{\pm}}} \|\nabla w\|_{\Omega_{\pm}}.
\end{equation}

Thus, (\ref{identity1})  and (\ref{estimate_on_M}) yield the estimate
\begin{equation}
\|\nabla w_{\lambda,q^*}\|_{\Omega} \leqslant C_{F_{\Omega_+}} \|\text{\rm{div}}\, q^*\|_{\Omega_+}
+C_{F_{\Omega_-}} \|\text{\rm{div}}\, q^*\|_{\Omega_-}+C_{Tr_{\mathcal{M}}}
\|\lambda - \left[ q^*\cdot \mathbf{n} \right]\|_{\mathcal{M}}.
\label{estimate_nabla_w}
\end{equation}
Notice that (\ref{identity1}) implies the identity
$$
\int\limits_{\Omega}\left(\nabla w_{\lambda,q^*}+q^*\right)\cdot \nabla \widetilde{w}dx=
\int\limits_{\mathcal{M}}\lambda \widetilde{w} d\mu  \qquad \forall \widetilde{w}\in H^1_0(\Omega),
$$
which shows that
the function $\tau^*:=\nabla w_{\lambda,q^*}+q^* \in Q^*_{\lambda, \mathcal{M}}$. Hence
$$
\inf\limits_{y^*\in Q^*_{\lambda, \mathcal{M}}}\|q^*-y^*\|_{\Omega} \leqslant \|q^*-\tau^*\|_{\Omega}=\|\nabla w_{\lambda,q^*}\|_{\Omega}.
$$
Now (\ref{replace_q}) follows from (\ref{estimate_nabla_w}).
\end{proof}
\vspace{0.2cm}

Let $q^* \in H(\Omega_{\pm}, \text{\rm{div}})$, and let $\lambda \in \Lambda$. For any $v\in \mathbb{K}$ and $y^*\in Q^{*}_{\lambda, \mathcal{M}}$ we have
\begin{equation}
\label{triangle_ineq}
\|\nabla v-y^*\|_{\Omega} \leqslant \|\nabla v-q^*\|_{\Omega}+\|q^*-y^*\|_{\Omega}.
\end{equation}
By  (\ref{27-a}), (\ref{triangle_ineq}), and (\ref{replace_q}), we  obtain the first advanced form of the error majorant:

\begin{thm} \label{main_theorem} 
For any $v\in \mathbb{K}$, the distance to the minimizer $u$ 
is subject to the estimate
\begin{equation} \label{estimate1}
\|\nabla (v-u)\|_{\Omega} \leqslant \mathfrak{M}(v, q^*, \lambda, \psi),
\end{equation}
where
$$
\begin{aligned}
\mathfrak{M}(v,q^*,\lambda,&\psi):=\|\nabla v -q^*\|_{\Omega}+\sqrt{2} \left(\,\int\limits_{\mathcal{M}}\lambda (v-\psi) d\mu \right)^{1/2}\\
&+C_{F_{\Omega_+}}\|\text{\rm{div}}\, q^*\|_{\Omega_+}+C_{F_{\Omega_-}}\|\text{\rm{div}}\, q^*\|_{\Omega_-}
+C_{Tr_{\mathcal{M}}}\|\lambda -[q^*\cdot \mathbf{n}]\|_{\mathcal{M}},
\end{aligned}
$$
$\lambda\in\Lambda$, $q^*$ is an arbitrary function in $H(\Omega_{\pm}, \text{\rm{div}})$,
$C_{F_{\Omega_+}}$, $C_{F_{\Omega_-}}$,  and 
$C_{Tr_{\mathcal{M}}}$ are the same constants as in Lemma~\ref{lemma2.2}.
\end{thm}

%Finally, a combination of (\ref{27-a}), (\ref{triangle-ineq}), (\ref{28}) and the Young inequality with a parameter $\beta>0$ %gives   the desired \textit{majorant estimate}:

In (\ref{estimate1}), the function $q^*$ is defined in a much wider set
of functions defined without differential relations. The majorant 
$\mathfrak{M}$ is a nonnegative functional, which vanishes if and only if
$v=u$ and $q^*=\nabla u$ almost everywhere in $\Omega$, and $\lambda=\lambda_*:=\left[\frac{\partial u}{\partial{\mathbf{n}}}\right]$ almost everywhere on $\mathcal{M}$.

\begin{rem}
\label{rem3}
By the same arguments as in Remark \ref{rem2}, we can prove that the majorant 
$\mathfrak{M}(v,q^*,\lambda,\psi)$ is sharp if
 ${\mathcal M}^u_\psi\subset {\mathcal M}^v_\psi$.
 \end{rem}

It is useful to have also  a modified version of (\ref{estimate1}), which follows from (\ref{27-a}), (\ref{triangle_ineq}),  and Young's inequalities (with the parameters $\beta_1$ and $\beta_2$).

\begin{cor}
For any $v\in \mathbb{K}$, $\beta_1>0$, $\beta_2>0$, $q^*\in H(\Omega_{\pm}, \text{\rm{div}})$, and $\lambda \in \Lambda$, we have
\begin{equation}
\|\nabla (v-u)\|^2_{\Omega} \leqslant \mathfrak{M}_1(v,q^*,\beta_1, \beta_2)+
\mathfrak{M}_2(v,q^*,\beta_1, \beta_2 \lambda, \psi),
\label{majorant_with_lambda_2}
\end{equation}
where
$$
\begin{aligned}
\mathfrak{M}_1(v, q^*, \beta_1, \beta_2):&=(1+\beta_1)\|\nabla v-q^*\|^2_{\Omega} \qquad \qquad \qquad \qquad\\
 +(1+\beta_1^{-1})&(1+\beta_2)\left[C_{F_{\Omega_+}}\|\text{\rm{div}}\, q^*\|_{\Omega_+}+C_{F_{\Omega_-}}\|\text{\rm{div}}\, q^*\|_{\Omega_-}\right]^2, \\
\mathfrak{M}_2(v,q^*,\beta_1, \beta_2, \lambda, \psi):&=(1+\beta_1^{-1})(1+\beta_2^{-1})
C^2_{Tr_{\mathcal{M}}}\|\lambda-  \left[q^* \cdot \mathbf{n}\right]\|^2_{\mathcal{M}}\\
&+2\int\limits_{\mathcal{M}}\lambda(v-\psi)
 d\mu ,
\end{aligned}
$$
and the constants $C_{F_{\Omega_+}}$, $C_{F_{\Omega_-}}$, and $C_{Tr_{\mathcal{M}}}$ are the same as in Lemma~\ref{lemma2.2}.
%whereas $q$ is an arbitrary function in $H(\Omega_{\pm}, \text{\rm{div}})$.
\end{cor}

The majorant (\ref{majorant_with_lambda_2}) contains parameters and free functions that can be selected arbitrarily in the respective sets. Below we deduce a new form  of  (\ref{majorant_with_lambda_2})  where the function $\lambda$ will be chosen in the optimal way.

First, we optimize $\mathfrak{M}_2$ with respect to $\lambda$.
The respective minimization problem is reduced to
$$
\inf\limits_{\lambda \in \Lambda} \mathfrak{M}_2=c^{-1}_{\beta}\inf\limits_{\lambda \in \Lambda}\int\limits_{\mathcal{M}}\left( C^2_{Tr_{\mathcal{M}}} (\lambda-  \left[q^* \cdot \mathbf{n}\right])^2 +  2\lambda c_{\beta} (v-\psi)\right) d\mu, 
$$
where $c_{\beta}:=\beta_1\beta_2(1+\beta_1)^{-1}(1+\beta_2)^{-1}$.
The corresponding Euler-Lagrange equation has the form
$$
2C^2_{Tr_{\mathcal{M}}}(\lambda-  \left[q^* \cdot \mathbf{n}\right])+2c_{\beta}(v-\psi)=0.
$$
From here,
taking into account the condition $\lambda \geqslant 0$ a.e. on $\mathcal{M}$,
we find the minimizer
$$
 \overline{\lambda}=\left\{
\begin{aligned}
\left[q^*\cdot \mathbf{n}\right] -& c_{\beta}C^{-2}_{Tr_{\mathcal{M}}}(v-\psi), \quad \ \text{if}\ \left[q^*\cdot \mathbf{n}\right] \geqslant 
c_{\beta}C^{-2}_{Tr_{\mathcal{M}}}(v-\psi),\\
&0, \qquad \qquad  \qquad \quad \ \text{if} \ \left[q^*\cdot \mathbf{n}\right] <
c_{\beta}C^{-2}_{Tr_{\mathcal{M}}}(v-\psi).
\end{aligned}
\right.
$$
Plugging $ \overline{\lambda}$ in the right-hand side of (\ref{majorant_with_lambda_2})
implies  the following result.
\begin{thm}
For any $v \in \mathbb{K}$, 
\begin{equation}
\|\nabla (v-u)\|^2_{\Omega} \leqslant 
\mathfrak{M}_1(v,q^*, \beta_1, \beta_2)+\mathfrak{M}_3(v,q^*, \beta_1, \beta_2, \psi),
\label{majorant_without_lambda}
\end{equation}
where $q^*$ is an arbitrary function in $H(\Omega_{\pm}, \text{\rm{div}})$,
$\beta_1$ and $\beta_2 $ are arbitrary nonegative numbers,
 $\mathfrak{M}_1$ is the same as in (\ref{majorant_with_lambda_2}),
$$
\mathfrak{M}_3(v,q^*, \beta_1, \beta_2,\psi):= \int\limits_{\mathcal{M}}\rho(v,q^*, c_{\beta}, \psi)
 d\mu ,
$$
and
$$
\rho(v,q^*, c_{\beta}, \psi):= \left\{
\begin{aligned}
(v-\psi) \big( 2 \left[q^*\cdot \mathbf{n}\right] -\frac{c_{\beta}}{C^{2}_{Tr_{\mathcal{M}}}}(v-\psi) \big), \  & \text{if} \ \left[q^*\cdot \mathbf{n}\right] \geqslant
\frac{c_{\beta}}{C^2_{Tr_{\mathcal{M}}}}(v-\psi), \\
\frac{C^2_{Tr_{\mathcal{M}}}}{c_{\beta}}\left[q^*\cdot \mathbf{n}\right]^2 , \qquad \qquad \ \ \  \  \quad & \text{if} \ \left[q^*\cdot \mathbf{n}\right] <
\frac{c_{\beta}}{C^2_{Tr_{\mathcal{M}}}}(v-\psi).
\end{aligned}
\right.
$$
\end{thm}
\vspace{0.2cm}

It is clear that the quantities $\mathfrak{M}_1$ and $\mathfrak{M}_3$ are always  nonnegative and the functional $\mathfrak{M}_4~:=~\mathfrak{M}_1~+\mathfrak{M}_3$ satisfies for any $\beta_1, \beta_2>0$ the relation
$$
\mathfrak{M}_4(u, \nabla u, \beta_1, \beta_2, \psi)=0.
$$
On the other hand, if $\mathfrak{M}_4(v, q^*, \beta_1, \beta_2, \psi)=0$ then $v=u$ almost everywhere in $\Omega$. Moreover, in this case the conditions
\begin{equation}
\begin{aligned}
q^*&=\nabla u\quad \text{a.e. in} \ \Omega,\\
\Delta u&=0 \qquad \text{a.e. in} \ \Omega_{\pm},\\
(u-\psi)\left[\nabla u \cdot \mathbf{n}\right]&=0 \qquad \text{a.e. on}\ \mathcal{M}
\end{aligned}
\label{xxx}
\end{equation}
hold true. We point out that the third equality in (\ref{xxx}) is provided by strict positivity of the factor $2 \left[q^*\cdot \mathbf{n}\right] -\frac{c_{\beta}}{C^{2}_{Tr_{\mathcal{M}}}}(v-\psi)$ in definition of $\rho$. Therefore, one can conclude that the majorant $\mathfrak{M}_4$ vanishes if and only if $v=u$ and $q^*=\nabla u$ almost everywhere in $\Omega$.

\begin{rem}
\label{rem3a}
Applying the same arguments as in Remark \ref{rem2}, we can prove that the majorant 
$\mathfrak{M}_4(v,q^*, \beta_1, \beta_2, \psi)$ is sharp if
 ${\mathcal M}^u_\psi\subset {\mathcal M}^v_\psi$.
 \end{rem}

One can also  prove that for any $\beta_1,\beta_2>0$,  the functional $\mathfrak{M}_4(v,q^*, \beta_1, \beta_2, \psi)$ possesses necessary continuity properties with respect to the first and second arguments. Thus, 
$$
\mathfrak{M}_4(v_k, q^*_k, \beta_1, \beta_2,\psi) \to 0
$$
if $v_k \to u$ in $\mathbb{K}$ and $q^*_k \to \nabla u$ in $L^2(\Omega_{\pm})$ and  $\left[q_k^*\cdot \mathbf{n}\right] \to \left[ \nabla u \cdot \mathbf{n}\right]$ in $L^2(\mathcal{M})$. Thus, taking into account Remark~\ref{rem3a}, we conclude that 
 the estimate (\ref{majorant_without_lambda}) has no gap between the left and right hand sides 
 and we can always select the parameters of $\mathfrak{M}_4$ such that it is arbitrary close to the energy norm of the error.
%\vspace{0.2cm}

%%%%%%%%%%%%%%%%%%%%%%%%%%%%%%%%%%%%%%%%%%%%%%%%%%%%%%%%%%%%%%%%%%%%%

\section{Estimates with explicit constants}
It should be also noted that for complicated domains the constants $C_{F_{\Omega_{\pm}}}$ and $C_{Tr_{\mathcal{M}}}(\Omega_{\pm})$ entering above derived estimates (\ref{estimate1})-(\ref{majorant_without_lambda})
may be unknown. In this case, we need to find guaranteed and realistic
upper bounds of them. Depending on a particular
domain, this task may be fairly easy or very difficult.
It is therefore of interest to look at other variants of Lemma 
\ref{lemma2.2}, which operates with  different constants.
In this section, we establish another estimate based
on the  Poincar\'{e}  inequlity for functions
having zero mean values in $\Omega_\pm$ and on the
so--called "sloshing" inequality for functions
with zero mean traces on $\mathcal{M}$.
As a result, we obtain estimates of the distance to the minimizer
$u$ containing the constants which are either explicitly known
or easily definable.

Henceforth, we denote by 
$\mean{w}_\omega$ the mean value of $w$ on the set $\omega$.
In view of the Poincare inequality
\begin{equation}
\label{Poincare1}
\|w\|_{\Omega_\pm}\leq C_{P_{\Omega_{\pm}}} \|\nabla w\|_{\Omega_\pm}\qquad
\forall w\in \widetilde H^1(\Omega_\pm):=\left\{w\in H^1(\Omega_\pm)\,:\,\mean{w}_{\Omega_\pm}=0\right\}.
\end{equation}
Similar inequalities hold for the  functions defined in $\Omega_+$
and $\Omega_-$  having zero mean values
on $\mathcal M$:
\begin{equation}
\label{Poincare2}
\|w\|_{\mathcal M}\leq C_{P_{\mathcal{M}}} (\Omega_{\pm}) \|\nabla w\|_{\Omega_\pm}\qquad
\forall w\in \widetilde H^1_{\mathcal M}(\Omega_\pm):=
\left\{w\in H^1(\Omega_\pm)\,:\,\mean{w}_{\mathcal M}=0\right\}.
\end{equation}

%We have to do this because, in general,  the values of $C_{F_{\Omega_{\pm}}}$ and, especially, $C_{Tr_{\mathcal{M}}}$ are difficult to find. Moreover, even getting the guaranteed upper bounds of these constants is a difficult problem which is equivalent to getting the lower bounds of respective minimal eigenvalues.

\begin{lem} \label{lemma_3.1}
Let $q^*\in H(\Omega_{\pm}, \text{\rm{div}})$ and $\lambda \in \Lambda$ satisfy  the following additional conditions:
\begin{equation}\label{zero_mean}
\int\limits_{\Omega_+}\text{\rm{div}}\,q^*dx=\int\limits_{\Omega_-}\text{\rm{\rm{div}}}\,q^*dx=0\qquad {\rm and}\qquad
\int\limits_{\mathcal{M}} (\lambda -\left[q^*\cdot \mathbf{n}\right]) d\mu =0.
\end{equation} 
 Then, for any $\alpha \in [0,1]$, we have
\begin{equation} \label{with_Poincare}
\inf\limits_{y^*\in Q^*_{\lambda,\mathcal{M}}}\|q^*-y^*\|^{2}_{\Omega} \leqslant 
(\mathfrak{D}_-(q^*)+\alpha \mathfrak{m}_-(q^*))^2+(\mathfrak{D}_+(q^*)+(1-\alpha)\mathfrak{m}_+(q^*))^2,
\end{equation}
where $
\mathfrak{D}_{\pm}(q^*):=C_{P_{\Omega_{\pm}}}\|\text{\rm{div}}\,q^*\|_{\Omega_{\pm}}
$ and
$\mathfrak{m}_{\pm}(q^*)=C_{P_{\mathcal{M}}}(\Omega_{\pm})\|\lambda -[q^*\cdot \mathbf{n}]\|_{\mathcal{M}}$.
% \rc{and}\,
%$C_{P_{\Omega_{\pm}}}$ and $C_{P_{\mathcal{M}}}(\Omega_{\pm})$ are the constants from the Poincare
%inequalities for $\Omega_{\pm}$ and from the Poincare type estimates for $M$, respectively.
%where the functionals $\mathfrak{D}_{\pm}$ and $\mathfrak{m}_{\pm}$ are the same as in Theorem~\ref{theorem_Poincare}.
\end{lem}

\begin{proof}
We use the same arguments as in the proof of Lemma~2.3 and arrive at the identity (\ref{identity1}). In view of (\ref{zero_mean}), this identity implies the relation
%Repeating word by word the beginning of the proof of Lemma~\ref{lemma2.2} we consider the problem ($\mathcal{P}_{q^*}$), denote by $w_{\lambda,q^*}$ the minimizer, and get the identity (\ref{identity1}).
%In contrast to Lemma~\ref{lemma2.2} we have now the additional conditions (\ref{zero_mean}). Thanks to them one can deduce from (\ref{identity1}) the relation
$$
\begin{aligned}
\|\nabla w_{\lambda,q^*}\|^2_{\Omega}&=\int\limits_{\Omega_+}\left(w_{\lambda,q^*}-\{w_{\lambda,q^*}\}_{\Omega_+}\right)\text{\rm{div}}\,q^*dx\\
&+\int\limits_{\Omega_-}\left(w_{\lambda,q^*}-\{w_{\lambda,q^*}\}_{\Omega_-}\right)\text{\rm{div}}\,q^*dx\\
&+\int\limits_{\mathcal{M}}(\lambda-\left[q^*\cdot \mathbf{n}\right])\left(w_{\lambda,q^*}-\{w_{\lambda,q^*}\}_{\mathcal{M}}\right) d\mu.
\end{aligned}
$$
By (\ref{Poincare1}) and (\ref{Poincare2}), we obtain
\begin{eqnarray*}
&&\int\limits_{\Omega_{\pm}}\left(w_{\lambda,q^*}-\{w_{\lambda,q^*}\}_{\Omega_+}\right)\text{\rm{div}}\,q^*dx \leqslant
C_{P_{\Omega_{\pm}}}\|\nabla w_{\lambda,q^*}\|_{\Omega_{\pm}}\|\text{\rm{div}}\,q^*\|_{\Omega_{\pm}},\\
&&\int\limits_{\mathcal{M}}(\lambda-\left[q^*\cdot \mathbf{n}\right])\left(w_{\lambda,q^*}-\{w_{\lambda,q^*}\}_{\mathcal{M}}\right) d\mu 
\leqslant C_{P_{\mathcal{M}}}(\Omega_{\pm})\|\nabla w_{\lambda,q^*}\|_{\Omega_{\pm}}\|\lambda -[q^*\cdot \mathbf{n}]\|_{\mathcal{M}}.
\end{eqnarray*}
Then,
\begin{equation} \label{Poincare_constants}
\begin{aligned}
\|\nabla w_{\lambda,q^*}\|^2_{\Omega}&\leqslant \mathfrak{D}_-(q^*)\|\nabla w_{\lambda,q^*}\|_{\Omega_-}+
\mathfrak{D}_+(q^*)\|\nabla w_{\lambda,q^*}\|_{\Omega_+}+\alpha \mathfrak{m}_-(q^*)\|\nabla w_{\lambda,q^*}\|_{\Omega_-}\\
&+(1-\alpha) \mathfrak{m}_+(q^*)\|\nabla w_{\lambda,q^*}\|_{\Omega_+}\\
&\leqslant \left( (\mathfrak{D}_-(q^*)+\alpha \mathfrak{m}_-(q^*))^2+(\mathfrak{D}_+(q^*)+(1-\alpha) \mathfrak{m}_+(q^*))^2\right)^{1/2}\|\nabla w_{\lambda,q^*}\|_{\Omega}.
\end{aligned}
\end{equation}

Using (\ref{Poincare_constants})  and repeating the same arguments as at the end of the proof of Lemma~\ref{lemma2.2}, we arrive at (\ref{with_Poincare}). 
\end{proof}

The quantities $\mathfrak{D}_{\pm}(q^*)$ contain the Poincar\'{e} constants for $\Omega_{\pm}$. If these domains are convex, then due to the estimate of Payne and Weinberger (see \cite{PW60}) we know that
%\rc{Reference to Payne Weinberger}
$$
C_{P_{\Omega_{\pm}}} \leqslant \frac{\text{diam}\, \Omega_{\pm}}{\pi}.
$$
The constants $C_{P_{\mathcal{M}}}(\Omega_{\pm})$ entering $\mathfrak{m}_{\pm}(q^*)$ are also easy to estimate. These constants are known for triangles (see \cite{NR15} and \cite{MR16}). Due to this fact,
we can easily obtain upper bounds of the constants for
a wide collection of domains.
\begin{figure}[ht]
\centering
\includegraphics[width=0.48\textwidth]{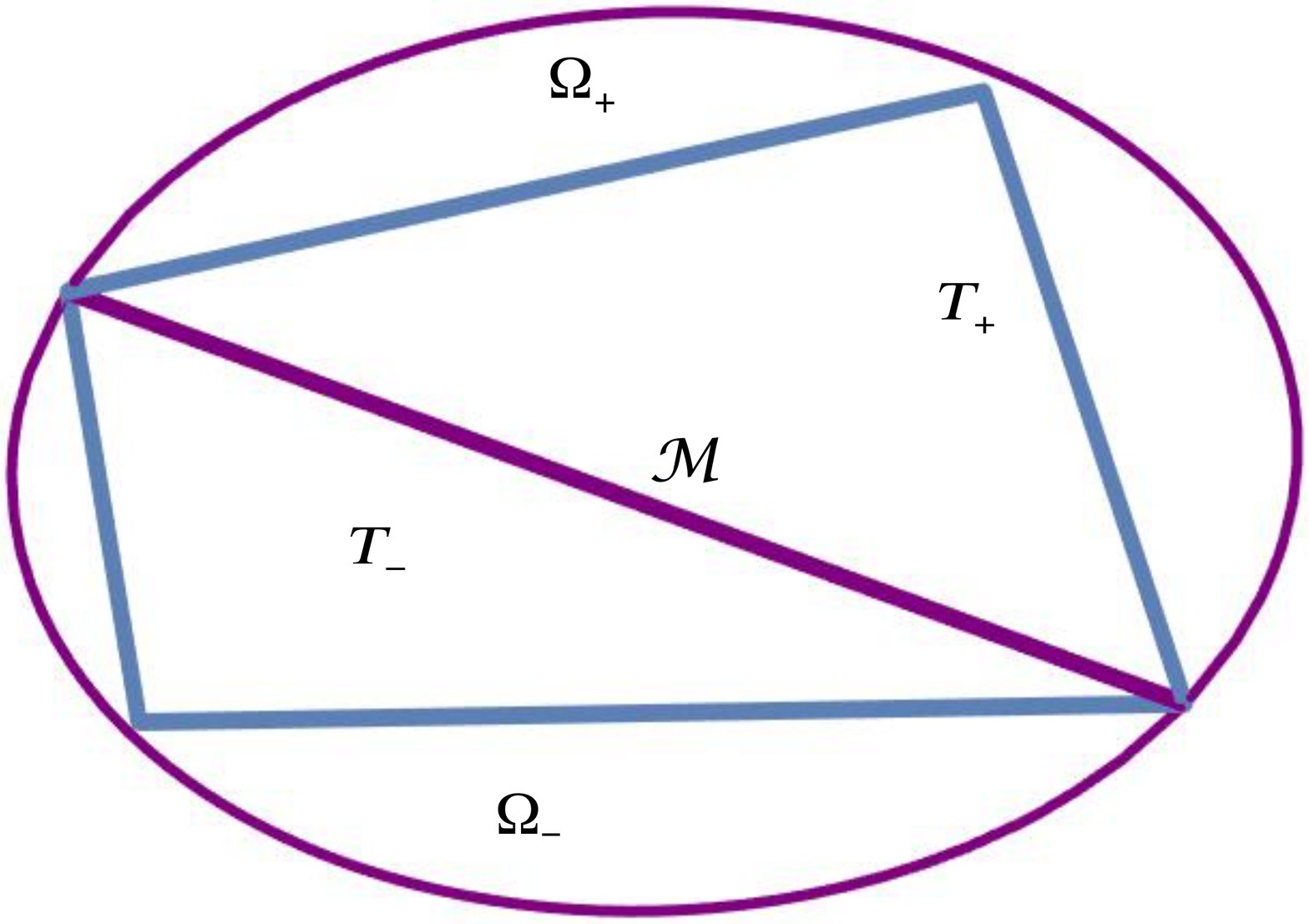}
\caption{Triangles $T_+$ and $T_-$}
\label{Triangles}
\end{figure}
\\
Indeed, let
$T_+ \subset \Omega_+$  and $\mathcal{M}$ as a face 
of this triangle (see Fig. 2).
It is clear that
$$
\|w\|_{\mathcal{M}} \leqslant C_{P_{\mathcal M }}(T_+)\|\nabla w\|_{T_+} \leqslant C_{P_{\mathcal M}}(T_+)\|\nabla w\|_{\Omega_+}\qquad
\forall w\in \widetilde H^1_{\mathcal M}(\Omega_+),
$$
and we can use the constant $C_{P_{\mathcal M}}(T_+)$ as
an upper bound of $C_{P_{\mathcal{M}}}(\Omega_{+})$.

Lemma~3.1, (\ref{27-a}), and (\ref{triangle_ineq}) yield the following majorant of the distance to $u$.
%With inequality~(\ref{with_Poincare}) at hands we argue along the same lines as at the end of Subsection~2.2. Namely, we put together  (\ref{27-a}), (\ref{triangle_ineq}), and (\ref{with_Poincare}). It results in the following version of the majorant estimate involving the Poincar\'{e} constants:

\begin{thm} \label{theorem_Poincare}
Let $u\in \mathbb{K}$ be a minimizer of  variational problem ($\mathcal{P}$). Let $q^* \in H(\Omega_+,\Omega_-, \text{\rm{div}})$, and let the conditions (\ref{zero_mean}) be satisfied.
Then, for any $v\in \mathbb{K}$, $\alpha \in [0,1]$, and $\lambda \in \Lambda$, the upper bound of error is given by the estimate
\begin{equation} \label{3}
\|\nabla (v-u)\|_{\Omega} \leqslant \mathfrak{M}_5(v, q^*, \alpha, \lambda, \psi),
\end{equation}
where 
$$
\begin{aligned}
\mathfrak{M}_5(v,q^*,\alpha,\lambda,\psi)&:=\|\nabla v -q^*\|_{\Omega}+\sqrt{2} \left(\int\limits_{\mathcal{M}}\lambda (v-\psi) d\mu \right)^{1/2}\\
&+\left((\mathfrak{D}_-(q^*)+\alpha \mathfrak{m}_-(q^*))^2+(\mathfrak{D}_+(q^*)+(1-\alpha)\mathfrak{m}_+(q^*))^2\right)^{1/2},
\end{aligned}
$$
where the functionals $\mathfrak{D}_{\pm}(q^*)$ and $\mathfrak{m}_{\pm}(q^*)$ are the same as in Lemma~\ref{lemma_3.1}.
\end{thm}

As in Section 2, it is easy to see that the majorant $\mathfrak{M}_5$  is a nonnegative function of its arguments, which  vanishes if and only if
$v=u$ and $q^*=\nabla u$ a. e. in $\Omega$, and $\lambda=\left[\frac{\partial u}{\partial{\mathbf{n}}}\right]$ a. e. on $\mathcal{M}$.

\begin{rem}
\label{rem4}
The majorant 
$\mathfrak{M}_5(v,q^*,\alpha,\lambda,\psi)$ is sharp if
 ${\mathcal M}^u_\psi\subset {\mathcal M}^v_\psi$. The proof is based on the same arguments as in Remark \ref{rem2}.
\end{rem}

\begin{rem} 
\label{Remark 1}
Other forms of the majorant arise if the conditions (\ref{zero_mean}) 
are satisfied only partially. For example, if  only the condition
\begin{equation} \label{partial_zero_mean}
\int\limits_{\mathcal{M}}(\lambda -\left[q^*\cdot \mathbf{n}\right] )d\mu =0
\end{equation}
is satisfied, then the estimate (\ref{3}) holds true for any $q^* \in H(\Omega_{\pm}, \text{\rm{div}})$ satisfying (\ref{partial_zero_mean}), where the
functionals $\mathcal{D}_{-}(q^*)$ and $\mathcal{D}_{+}(q^*)$ in $\mathfrak{M}_3$ are replaced by
$C_{F_{\Omega_{-}}}\|\text{\rm{div}}\,q^*\|_{\Omega_{-}}$
and $C_{F_{\Omega_{+}}}\|\text{\rm{div}}\,q^*\|_{\Omega_{+}}$, respectively. 
This version of the estimate is used in the examples considered in Section~5.
\end{rem}

Obviously, if $\mathfrak{m}_+(q^*)=\mathfrak{m}_-(q^*)=0$ then the parameter $\alpha$ in (\ref{3}) has no influence to the majorant value and it  can be chosen arbitrarily in $[0,1]$. Otherwise,
  we can define  $\alpha$  in the optimal way by solving the
minimization problem
$$
\inf\limits_{\alpha \in [0,1]} \left\{ (\mathfrak{D}_-(q^*)+\alpha\mathfrak{m}_-(q^*))^2+(\mathfrak{D}_+(q^*)+(1-\alpha)\mathfrak{m}_+(q^*))^2\right\},
$$
which yields the best value
\begin{equation*}
\alpha^*:=\left\{\begin{array}{ll}
&\overline{\alpha}, \quad {\rm if}\;0\leqslant \overline{\alpha} \leqslant 1,\\
&0, \quad {\rm if} \; \overline{\alpha}<0,\\
&1, \quad {\rm if}\; \overline{\alpha}>1,
\end{array}
\right.\qquad{\rm where}\qquad \overline{\alpha}:=\frac{\mathfrak{m}^2_+(q^*)+\mathfrak{D}_+(q^*)\mathfrak{m}_+
(q^*)-\mathfrak{D}_-(q^*)\mathfrak{m}_-(q^*)}{\mathfrak{m}^2_+(q^*)+\mathfrak{m}^2_-(q^*)}.
\end{equation*}

%%%%%%%%%%%%%%%%%%%%%%%%%%%%%%%%%%%%%%%%%%%%%%%%%%
\section{The scalar Signorini problem}

A problem close to $(\mathcal{P})$ arises if $\mathcal{M}$ 
coincides with a part of $\partial\Omega$. In this case, the functional
(\ref{Dint}) is minimized over  the set
$$
\mathbb{K}_{\mathcal{S}}=\left\lbrace v \in H^{1}\left( \Omega\right) : \quad v \geqslant \psi\ \text{on}\ \mathcal{M}, \quad v = \varphi\ \text{on}\ \partial\Omega \setminus \mathcal{M}\right\rbrace .
$$
This problem is known as the \textit{boundary thin obstacle problem} or the \textit{(scalar) Signorini problem}.

Under appropriate assumptions
 on the data of the problem $(\mathcal{S})$, the existence of the unique minimizer $u \in H^1(\Omega)$   has been proved in \cite{F63}.
The exact solution $u$ is a harmonic function in $\Omega$, 
which satisfies the so-called \textit{Signorini boundary conditions}
$$
u-\psi \geqslant 0, \qquad \frac{\partial u}{\partial \mathbf{n}} \geqslant 0, \qquad (u-\psi)\frac{\partial u}{\partial \mathbf{n}}=0\qquad
{\rm on}\quad \mathcal{M},
$$
where $\mathbf{n}$ denotes the  unit outward normal to $\partial \Omega$.

Throughout this section 
$H^1_{0,{\mathcal{S}}}(\Omega)$ denotes a subset of $H^1(\Omega)$ containing the functions with zero traces on $\partial\Omega \setminus \mathcal{M}$ and 
$$
Q^{*, \, \mathcal{S}}_{\lambda, \mathcal{M}}:=\big\lbrace y^* \in L^2(\Omega, \mathbb{R}^n) : \int\limits_{\Omega}y^* \cdot \nabla w dx=\int\limits_{\mathcal{M}}\lambda w  d\mu  \ \text{for all}\ w\in H^1_{0,{\mathcal{S}}}\left( \Omega\right) \big\rbrace .
$$
Repeating all the arguments used in the derivation of (\ref{27-a})
(where $H^1_0(\Omega)$ is replaced by $H^1_{0,{\mathcal{S}}}(\Omega)$), we conclude that the estimate
\begin{equation}
\frac{1}{2}\|\nabla (v-u)\|_{\Omega}^{2}\leqslant \frac{1}{2}\|\nabla v-y^*\|^2_{\Omega} +\int\limits_{\mathcal{M}}\lambda \left( v-\psi\right)  d\mu 
\label{27-S}
\end{equation}
holds true for all $v\in \mathbb{K}_{\mathcal{S}}$, all $\lambda \in \Lambda$, and all $y^*\in Q^{*, \, \mathcal{S}}_{\lambda, \mathcal{M}}$.

The estimate (\ref{27-S}) can be extended to a wider set of functions
by the arguments similar to those used in Sect. 2. For this purpose,
we consider an auxiliary problem ($\mathcal{P}^{\mathcal{S}}_{q^*}$):
find $w^S_{\lambda,q^*}\in H^1_{0,{\mathcal{S}}}(\Omega)$ that minimizes the functional
$$
\mathcal{J}_{q^*}(w)=\int\limits_{\Omega}\left(\frac{1}{2}|\nabla w|^2+q^*\cdot \nabla w\right)dx -\int\limits_{\mathcal{M}}\lambda w  d\mu 
%\eqno{(\mathcal{P}_{q^*})}
$$
for a given
 $q^* \in H^{{\mathcal S}}(\Omega, \text{\rm{div}}):=\left\{q^*\in L^2(\Omega, {\mathbb R}^n)\,\mid\, {\rm div}\, q^*\in L^2(\Omega),\; \left[q^*\cdot{\bf n}\right]\in L^2({\mathcal M})\right\}$.

 By the same arguments as in Subsection 2.2, we conclude that the problem $(\mathcal{P}^{\mathcal{S}}_{q^*})$ has a unique minimizer $w_{\lambda,q^*}^S$ in $ H^1_{0,{\mathcal{S}}}(\Omega)$. In view of the respective integral identity, the function
$$\tau_{\mathcal{S}}^*(x):=\nabla w^{\mathcal{S}}_{\lambda, q^*}(x)+q^*(x)$$ belongs to the set $Q^{*,\mathcal{S}}_{\lambda, \mathcal{M}}$.
Hence 
\begin{equation}
\begin{aligned}
\inf\limits_{y^*\in Q^{*, \mathcal{S}}_{\lambda, \mathcal{M}}}\|\nabla v-y^*\|_{\Omega} &\leqslant \|\nabla v-q^*\|_{\Omega}+
\inf\limits_{y^*\in Q^{*, \mathcal{S}}_{\lambda, \mathcal{M}}}\|q^*-y^*\|_{\Omega}\\
& \leqslant \|\nabla v-q^*\|_{\Omega}+\|q^*-\tau^*_S\|_{\Omega}\\
\leqslant \|\nabla v-q^*\|_{\Omega} &+ C_{F_{\Omega}}\|\text{\rm{div}}\, q^*\|_{\Omega}+C_{Tr_{\mathcal{M}}}\|\lambda -q^*\cdot \mathbf{n}\|_{\mathcal{M}}.
\end{aligned}
\label{triangle-S}
\end{equation}
 for  any $v\in \mathbb{K}_{\mathcal{S}}$, $q^*\in H^{\mathcal S}(\Omega, \text{\rm{div}})$, and $\lambda \in \Lambda$. Here $C_{F_{\Omega}}$ and $C_{Tr_{\mathcal{M}}}$
 are constants in is the the Friedrichs and trace inequalities,
 respectively.
 
Combining (\ref{27-S}) and (\ref{triangle-S}), we find that for any $v\in \mathbb{K}_{\mathcal{S}}$ the following estimate holds
\begin{equation}
\|\nabla (v-u)\|_{\Omega} \leqslant \mathfrak{M}^{\mathcal{S}} (v,q^*, \lambda, \psi),
\label{majorant_S_with_lambda}
\end{equation}
where
\begin{multline}
\mathfrak{M}^{\mathcal{S}}(v, q^*, \lambda, \psi):=\|\nabla v-q^*\|^2_{\Omega} +\sqrt{2}
\left(\int\limits_{\mathcal{M}}\lambda(v-\psi)
 d\mu \right)^{1/2}\\
 +C_{F_{\Omega}}\|\text{\rm{div}}\, q^*\|_{\Omega}+
C_{Tr_{\mathcal{M}}}\|\lambda-  q^* \cdot \mathbf{n}\|_{\mathcal{M}},
\end{multline}
$\lambda \in \Lambda$, and $q^*$ is an arbitrary function in $H^{\mathcal S}(\Omega, \text{\rm{div}})$.
\vspace{0.2cm}

Obviously,  $\mathfrak{M}^{\mathcal{S}}(v, q^*, \lambda, \psi)\geq 0$. \ By \ the \ same \ arguments \ as \ in \ Section 2,
we prove that $\mathfrak{M}^{\mathcal{S}}(v, q^*, \lambda, \psi)= 0$  if and only if $v=u$,
$q^*=\nabla u$ a. e. in $\Omega$, and $\lambda = \frac{\partial u}{\partial \mathbf{n}}$ a. e. on $\mathcal{M}$.

Assume  that a function $q^*\in H^{\mathcal S}(\Omega, \text{\rm{div}})$ additionally  satisfies the conditions
$$
\int\limits_{\Omega} \text{\rm{div}}\, q^*dx=0 \qquad{\rm and}\qquad \int\limits_{\mathcal{M}}\left(\lambda -q^*\cdot \mathbf{n}\right) d\mu =0.
$$

Then, we obtain the following analog of the estimate derived
in 
 Section 3:
$$
\|\nabla (v-u)\|_{\Omega} \leqslant \mathfrak{M}_1^{\mathcal{S}}(v, q^*, \lambda, \psi),
$$
where 
$$
\begin{aligned}
\mathfrak{M}_1^{\mathcal{S}}(v, q^*, \lambda, \psi)&:=\|\nabla v-q^*\|^2_{\Omega} +\sqrt{2}
\left(\int\limits_{\mathcal{M}}\lambda(v-\psi)
 d\mu \right)^{1/2}\\
&\ +C_{P_{\Omega}}\|\text{\rm{div}}\, q^*\|_{\Omega}+
C_{P_{\mathcal{M}}}\|\lambda-  q^* \cdot \mathbf{n}\|_{\mathcal{M}}.
\end{aligned}
$$
Here $\lambda$ is any function from $\Lambda$ and $C_{P_{\Omega}}$ and $C_{P_{\mathcal{M}}}$ are the constants from the Poincar\'{e} type inequalities for $\Omega$ and  for  $\mathcal{M}$, respectively.
It is not difficult to show that  $\mathfrak{M}_1^{\mathcal{S}}$ is nonnegative and vanishes if and only if $v=u$ and $q^*=\nabla u$ a. e. in $\Omega$, and $\lambda = \frac{\partial u}{\partial \mathbf{n}}$ a. e. on $\mathcal{M}$.

%%%%%%%%%%%%%%%%%%
\section{Examples}
Let $\overline{\Omega}=\overline{\Omega}_+\cup \overline{\Omega}_-$, where $\Omega_{\pm}$ are two right triangles having a common face
  $\mathcal{M}:=\{x_2=0\}$ (see Fig. 3). 
  \vskip-0.5cm
\begin{figure}[H]
\centering
\includegraphics[width=0.98\textwidth]{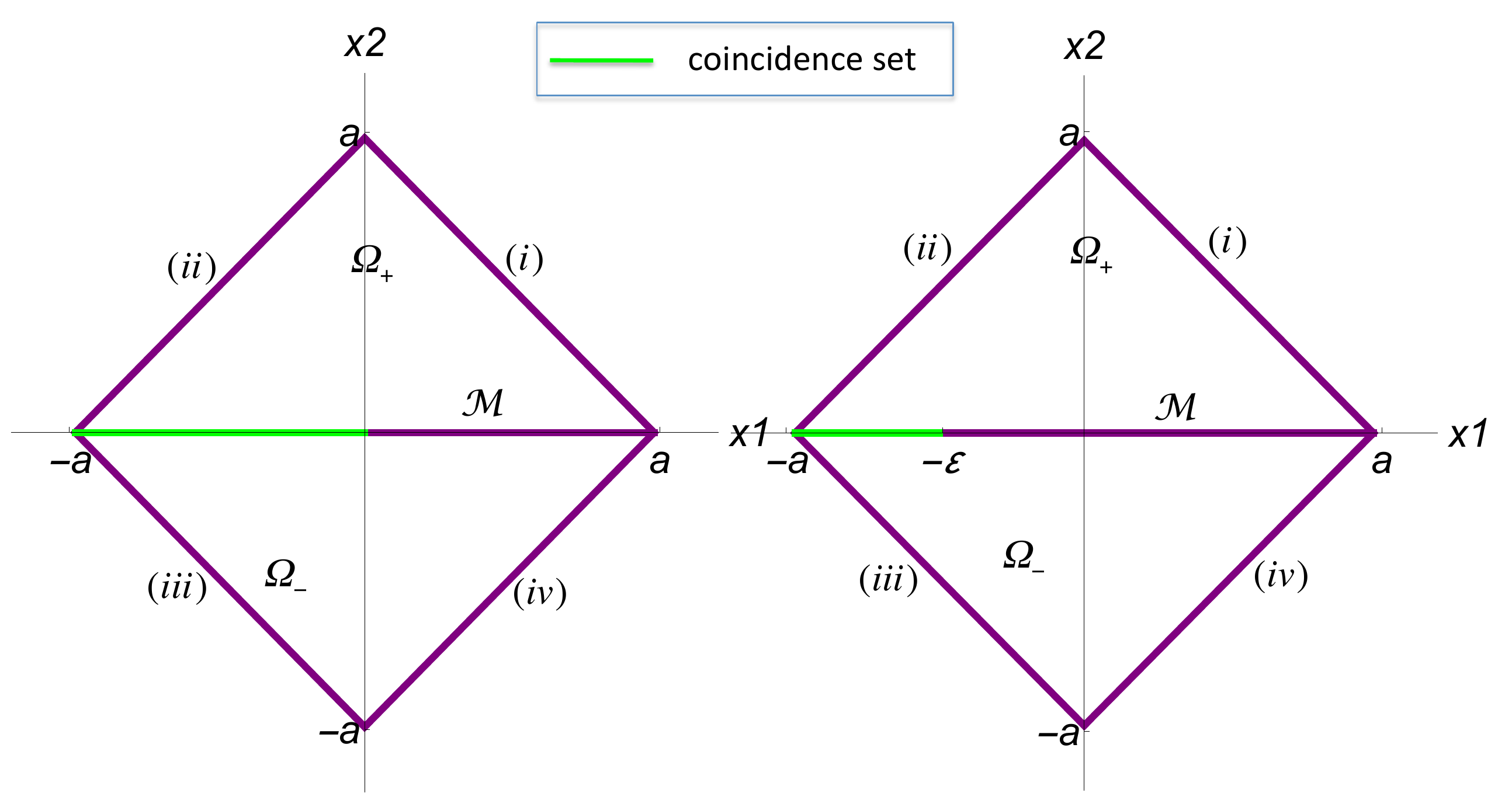}
\caption{Domains $\Omega_+$ and $\Omega_-$; coincidence sets for $u$ and $v_1$ (left) and for $v_{3,\varepsilon}$ (right)}
\label{Example-Omega}
\end{figure} 
\noindent
In this example, $\partial \Omega$ consists of four parts:
\begin{center}
\begin{tabular}{lrclr}
$(i)$ & $x_1+x_2-a=0$, &\qquad \qquad \qquad  &$(iii)$ &  $-x_1-x_2-a=0$,\\
$(ii)$ & $-x_1+x_2-a=0$, &\qquad \qquad \qquad  &$(iv)$ & $x_1-x_2-a=0$.
\end{tabular}
\end{center}

Notice that for this example, we can explicitly define the minimizer.
It is well known (see \cite{PSU12}) that for all $R >0$
$$
u(x_1,x_2)=\text{\rm{Re}} \left((x_1+i|x_2|)^{3/2}\right)
$$ is the exact solution of the thin obstacle problem in  $B_R \subset \mathbb{R}^2$ with $\mathcal{M}:=\{x_2=0\}$, and $\psi \equiv 0$, and  $\varphi=u\big|_{\partial B_R}$. Here, $B_R$ denotes the open ball with  center at the origin
 and radius $R$.  It is clear that $\Delta u=0$ in $\Omega_{\pm}$. In addition,
$$
u(x_1,0)=\left\{ \begin{aligned}
&0, \qquad \text{if}\ x_1\leqslant 0,\\
&x_1^{3/2}, \quad \text{if}\ x_1>0 
\end{aligned}\right. \qquad \text{and} \qquad 
\left[\frac{\partial u}{\partial \mathbf{n}} \right]=\left\{ \begin{aligned}
&3\sqrt{-x_1}, \quad \text{if}\ x_1<0,\\
&\ \ 0, \qquad \ \  \text{if}\ x_1\geqslant 0. 
\end{aligned}\right.
$$
Setting the boundary condition $\varphi$ on $\partial\Omega$ as the trace of $\text{\rm{Re}} \left((x_1+i|x_2|)^{3/2}\right)$ and taking $\psi \equiv 0$, we see that the restriction
of $u$ to $\Omega$ is the exact solution of the thin obstacle problem in this bounded domain  (see Fig.~\ref{Exact-Solution-bild} (left)).

\setcounter{figure}{3}
\begin{figure}[htbp]
\centering
\includegraphics[width=0.48\textwidth]{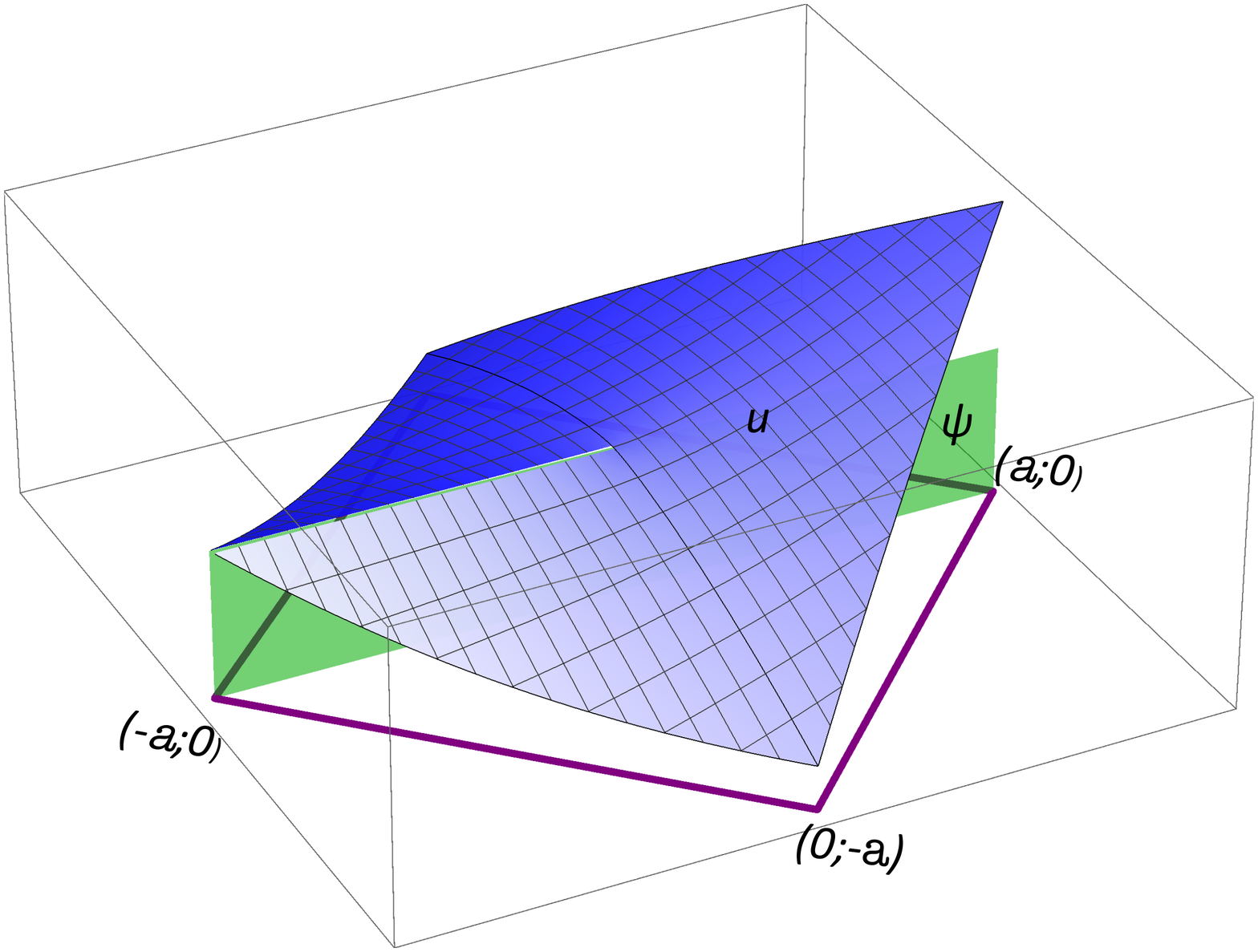}\qquad
\includegraphics[width=0.38\textwidth]{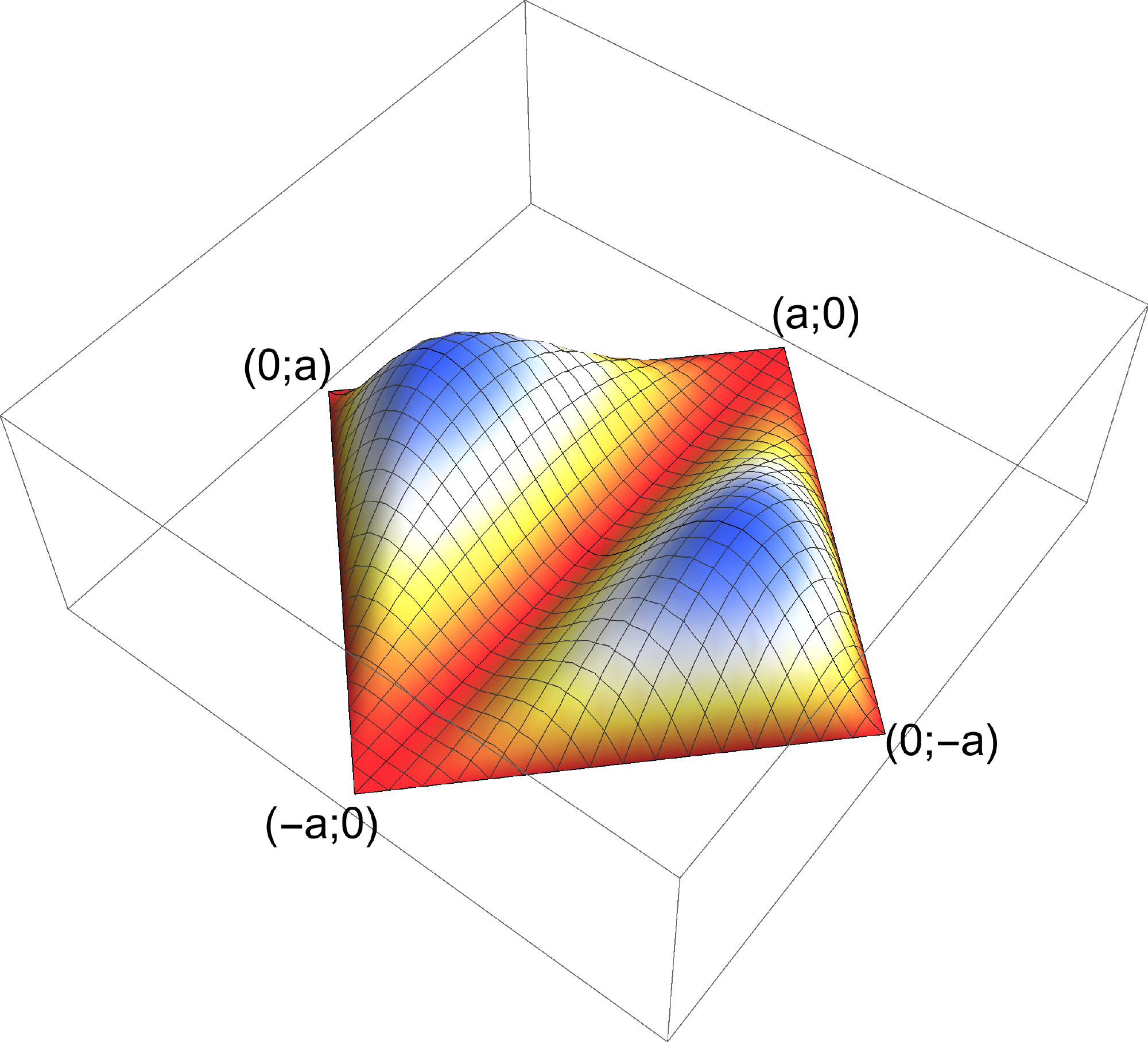}
\caption{The exact solution $u$ (left) and the function 
$v_1-u$ (right)}
\label{Exact-Solution-bild}
\end{figure}

In order to verify the performance of our estimates, we select
different functions $v$ in $\mathbb K$ and compute the distances
between $v$ and $u$.

\begin{exa} \label{example-1}

First, we define
$v=v_1$ as follows:
$$
v_1(x_1,x_2):=u(x_1,x_2)+\left\{\begin{aligned}
&x_2^2(x_2-x_1-a)(x_2+x_1-a), \quad \text{if}\ x_2 \geqslant 0,\\
&x_2^2(x_2-x_1+a)(x_2+x_1+a), \quad \text{if}\ x_2 < 0.
\end{aligned} \right.
$$
It is clear that $v_1 \in \mathbb{K}$ and $v_1 \geqslant u$ in $\Omega$ and $v_1(x_1,0)=u(x_1,0)$. Thus, $v_1$ has the same coincidence set as the exact solution $u$ (see Fig.~\ref{Example-Omega} (left) and Fig.~\ref{Exact-Solution-bild} (right)). 

By direct computations, we find  that
$\left[\frac{\partial v_1}{\partial \mathbf{n}}\right]=\left[\frac{\partial u}{\partial \mathbf{n}}\right]+12ax_2^2$, 
\begin{align*}
\Delta v_1&=\left\{\begin{aligned}
&10x_2^2-12x_2a-2x_1^2+2a^2, \quad \text{in}\ \Omega_+\\
&10x_2^2+12x_2a-2x_1^2+2a^2, \quad \text{in}\ \Omega_-
\end{aligned}\right.,
\end{align*}
and the exact error
$$
\|\nabla (v_1-u)\|_{\Omega}=\frac{4}{3}\sqrt{\frac{2}{35}}a^4.
$$

Let us set here $q^*=\nabla v_1$ and $\lambda=\left[\frac{\partial v_1}{\partial \mathbf{n}}\right]$. Computing the majorant $\mathfrak{M}\left( v_1 , \nabla v_1, \left[\frac{\partial v_1}{\partial \mathbf{n}}\right],0\right)$, defined by 
(\ref{estimate1}), 
%and taking into account (\ref{6.1}), 
we get
\begin{equation}
\begin{aligned}
\mathfrak{M}\left( v_1 , \nabla v_1, \left[\frac{\partial v_1}{\partial \mathbf{n}}\right],0\right)&=
\|\nabla  v_1 -
\nabla v_1\|_{\Omega}+\sqrt{2} \left( \int\limits_{-a}^a\left[\frac{\partial v_1}{\partial \mathbf{n}}\right]  v_1 dx\right)^{1/2}\\
&+C_{F_{\Omega_+}}\|\text{\rm{div}}\, \nabla v_1\|_{\Omega_+}+C_{F_{\Omega_-}}\|\text{\rm{div}}\, \nabla v_1\|_{\Omega_-}\\\
&+C_{Tr_{\mathcal{M}}}\|\left[\frac{\partial v_1}{\partial \mathbf{n}}\right]-\left[\nabla v_1\cdot \mathbf{n}\right]\|_{\mathcal{M}}\\
&=C_{F_{\Omega_+}}\|\Delta v_1\|_{\Omega_+}+C_{F_{\Omega_-}}\|\Delta v_1\|_{\Omega_-}.
%&=\frac{a}{\pi}\left(\|\Delta v_1\|_{\Omega_+}+\|\Delta v_1\|_{\Omega_-}\right)=\frac{16}{3\sqrt{5} \pi}a^4.
\end{aligned}
\label{majorant-for-v1}
\end{equation}
\end{exa}
\medskip

\begin{rem}
Here the constants $C_{F_{\Omega_+}}$ and $C_{F_{\Omega_-}}$ are defined by the quotient type relations
$$
\inf\limits_{w\in H^{1}_{0,\pm}(\Omega_{\pm})} \frac{\|\nabla w\|_{\Omega_{\pm}}}{\|w\|_{\Omega_{\pm}}}=
\frac{1}{C_{F_{\Omega_{\pm}}}},
$$
where $H^{1}_{0,+}(\Omega_+)$ contains all $H^1$-functions vanishing on $(i)$ and $(ii)$ and $H^{1}_{0,-}(\Omega_-)$ contains all $H^1$-functions vanishing on $(iii)$ and $(iv)$.
It is easy to show that
\begin{equation}
 \label{6.1}
C_{F_{\Omega_+}}=C_{F_{\Omega_-}}=\frac{a}{\pi}.
\end{equation}
Indeed, consider the rotated triangle (see Fig. 5) and the respective eigenvalue problem
\begin{center}
\begin{tabular}{p{5.5cm}c}
{\begin{align*}
\Delta w +\varkappa w&=0 \quad  \text{in}\quad \Omega_+,\\
w&=0\quad \text{on}\quad \widetilde{x}_1=0,\\
w&=0\quad \text{on}\quad \widetilde{x}_2=0,\\
\frac{\partial w}{\partial\mathbf{n}}&=0 \quad \text{on} \quad \mathcal{M},\\
\mathcal{M}:=&\{\widetilde{x}_1+\widetilde{x}_2=a\sqrt{2}\},
\end{align*}}
&
\raisebox{-130pt}
{
\includegraphics[width=0.5\textwidth]{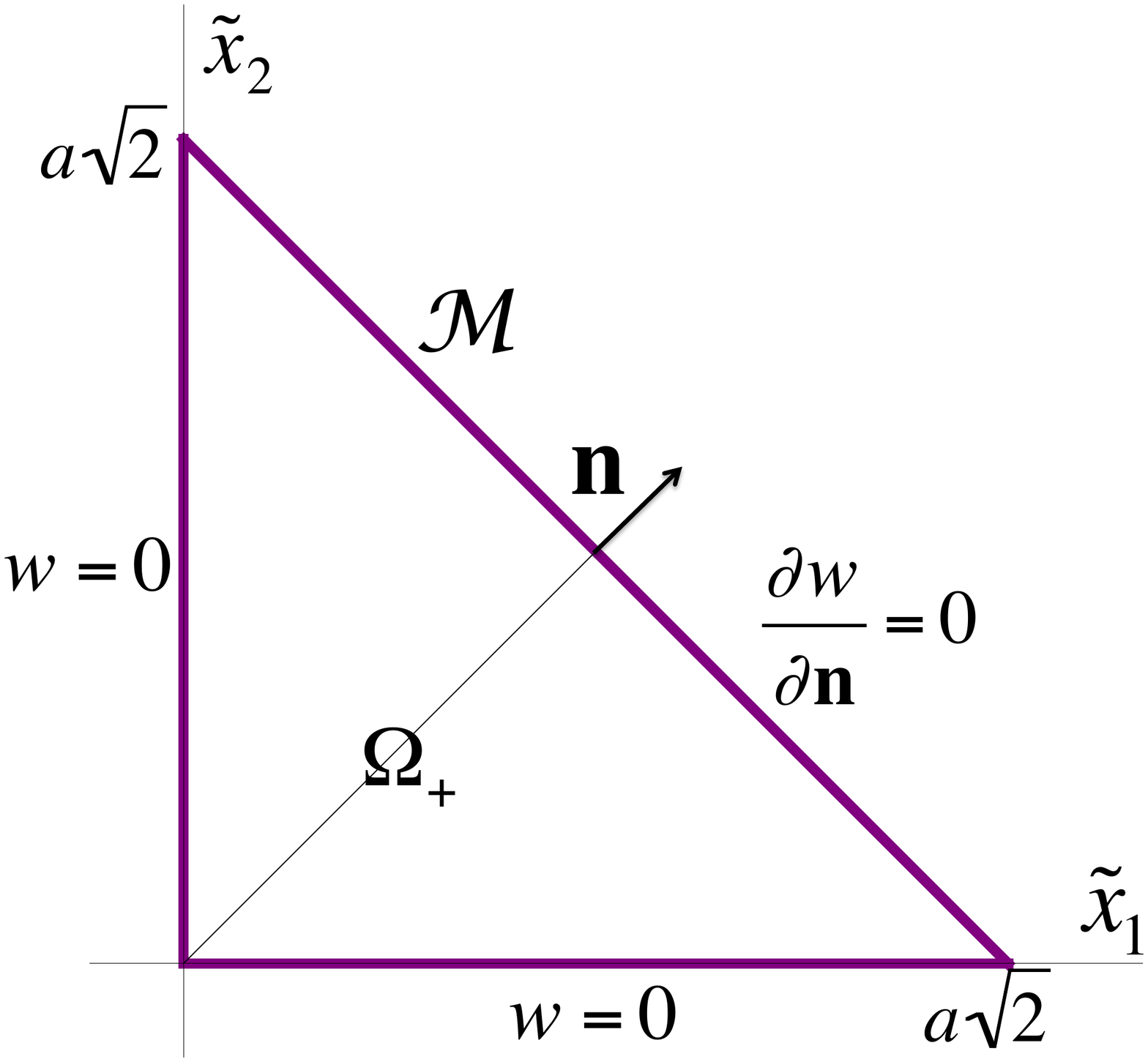}
} \\
\quad & \textsl{Figure 5: Eigenvalue problem }
\end{tabular}
\end{center}
Figure 5 is referred to the eigenvalue problem where the minimal eigenvalue corresponds to the eigenfunction
$$
\widetilde{w}=\sin{\left(\frac{\pi}{a\sqrt{2}}\widetilde{x}_1\right)}\sin{\left(\frac{\pi}{a\sqrt{2}}\widetilde{x}_2\right)}.
$$
Direct calculation of $\|\nabla \widetilde{w}\|_{\Omega_{\pm}}$ and $\|\widetilde{w}\|_{\Omega_{\pm}}$ yields (\ref{6.1}).
\label{REx1}
\end{rem}

Plugging (\ref{6.1}) into (\ref{majorant-for-v1}) yields the equality
$$
\mathfrak{M}\left( v_1 , \nabla v_1, \left[\frac{\partial v_1}{\partial \mathbf{n}}\right],0\right)=
\frac{16}{3\sqrt{5} \pi}a^4.
$$
Therefore,
the efficiency of the estimate is characterized by the value (efficiency index)
$$
1 \leqslant \frac{\mathfrak{M}\left( v_1 , \nabla v_1, \left[\frac{\partial v_1}{\partial \mathbf{n}}\right],0\right)}{\|\nabla (v_1-u)\|_{\Omega}}=\frac{4}{\pi}\sqrt{\frac{7}{2}}\approx 2.382.
$$
\vspace{0.2cm}

It should be pointed out that for $q^*=\nabla v_1$ and $\lambda=\left[\frac{\partial v_1}{\partial \mathbf{n}}\right]$ the assumption (\ref{partial_zero_mean}) is fulfilled. Moreover, $\|\left[\frac{\partial v_1}{\partial \mathbf{n}}\right]-\left[\nabla v_1\cdot \mathbf{n}\right]\|_{\mathcal{M}}=0$.
Thus, for any $\alpha \in [0,1]$ we can also compute a version of the majorant $\mathfrak{M}_3\left(v_1, \nabla v_1, \alpha, \left[\frac{\partial v_1}{\partial \mathbf{n}}\right],0\right)$, which is modified in accordance with Remark~\ref{Remark 1}. 
We will denote this modified majorant by $\mathfrak{M}'_3$. Taking into account (\ref{6.1}), we get
$$
\begin{aligned}
\mathfrak{M}'_3\left( v_1 , \nabla v_1, \alpha, \left[\frac{\partial v_1}{\partial \mathbf{n}}\right],0\right)&=
\|\nabla  v_1 -
\nabla v_1\|_{\Omega}+\sqrt{2} \left( \int\limits_{-a}^a\left[\frac{\partial v_1}{\partial \mathbf{n}}\right]  v_1 dx\right)^{1/2}\\
&+\left[C^2_{F_{\Omega_-}}\|\text{\rm{div}}\, \nabla v_1\|^2_{\Omega_-}+
C^2_{F_{\Omega_+}}\|\text{\rm{div}}\, \nabla v_1\|^2_{\Omega_+}\right]^{1/2}\\
&=\frac{a}{\pi}\left[ \|\Delta v_1\|^2_{\Omega_-}+\|\Delta v_1\|^2_{\Omega_+}\right]^{1/2}=\frac{8\sqrt{2}}{3\sqrt{5}\pi}a^4.
\end{aligned}
$$
Hence we have better efficiency index
$$
1 \leqslant \frac{\mathfrak{M}'_3\left(v_1, \nabla v_1, \alpha, \left[\frac{\partial v_1}{\partial \mathbf{n}}\right],0\right)}{\|\nabla (v_1-u)\|_{\Omega}}=\frac{2}{\pi}\sqrt{7}\approx 1.684.
$$

Finally, we notice  that in view of Remark~\ref{rem3}, the majorant $\mathfrak{M}$ is sharp for $q^*=\nabla u$ and $\lambda=\left[\frac{\partial u}{\partial \mathbf{n}}\right]$, i.e., 
$$
\frac{\mathfrak{M}(v_1, \nabla u, \left[\frac{\partial u}{\partial \mathbf{n}}\right],0)}{\|\nabla (v_1-u)\|_{\Omega}}=1.
$$
\medskip

\begin{exa} \label{example-2}
Consider now another function $v=v_2$, where
$$
 v_2 (x_1,x_2):=u(x_1,x_2)+(x_1+x_2-a)(x_2-x_1-a)(x_1+x_2+a)(x_2-x_1+a).
$$
Obviously,   $ v_2  > u$ in $\Omega$.  Hence $ v_2  \in \mathbb{K}$ and the respective coincidence set is empty.

Moreover, we have $\Delta  v_2 =8(x_1^2+x_2^2-a^2)$ in $\Omega_{\pm}$, and 
\begin{equation} \label{6.3}
\left[\frac{\partial{ v_2 }}{\partial \mathbf{n}}\right]=\left[\frac{\partial u}{\partial \mathbf{n}}\right]=\left\{ \begin{aligned}
&3\sqrt{-x_1}, \quad \text{if}\ x_1<0,\\
&\ \ 0, \qquad \ \  \text{if}\ x_1\geqslant 0. 
\end{aligned}\right.
\end{equation}
Let $q^*=\nabla  v_2 $ and $\lambda=\left[\frac{\partial  v_2 }{\partial\mathbf{n}}\right]$. Then the  assumption (\ref{partial_zero_mean}) is satisfied. We take into account (\ref{6.1}), Remark~\ref{Remark 1} and apply use the estimate (\ref{3}), which gives
$$
\begin{aligned}
\|\nabla( v_2 -u)\|_{\Omega}=\frac{16}{3\sqrt{5}}a^4&\leqslant \sqrt{2}\left(\int\limits_{-a}^0 3\sqrt{-x_1}\, v (x_1,0)dx_1\right)^{1/2}\\
&+\frac{a}{\pi}\left(
\|\Delta v_2 \|^2_{\Omega_+}
+
\|\Delta v_2 \|^2_{\Omega_-}\right)^{1/2}.
\end{aligned}
$$
This estimate has the  efficiency index
$$
1\leq \frac{\mathfrak{M}_3\left( v , \nabla v_2, \left[\frac{\partial v_2}{\partial \mathbf{n}}\right],0\right)}{\|\nabla( v_2 -u)\|_{\Omega}}\leq
\frac{\sqrt{22}}{\pi}+\sqrt{\frac{45}{2\cdot 77}}a^{-5/4}\approx
1.493+0.541 a^{-5/4},
$$
which shows that the upper bound is quite realistic.
\end{exa}
\vspace{0.1cm}

In the next example, we study the behavior of error majorants for some sequences of the approximate solutions $(v_{3,\varepsilon} )\subset \mathbb{K}$, which converges to the exact solution $u$ as $\varepsilon \to 0$.
\medskip

\begin{exa} \label{example-3}
Let $v=v_{3,\varepsilon}$ be defined as follows:
$$
v_{3,\varepsilon}(x_1,x_2):=u(x_1,x_2)+\varepsilon^2 \left\{
\begin{aligned}
&0, \qquad \qquad \qquad \qquad \qquad \qquad \qquad   \text{if}\  (x_1, x_2) \in \Omega,\ x_1 \leqslant -\varepsilon,\\
&\beta (x_1) (a+x_1+x_2) (a+x_1-x_2), \quad \text{if}\  (x_1, x_2) \in \Omega,\ -\varepsilon<x_1\leqslant 0,\\
&\beta (x_1) (a-x_1-x_2)(a-x_1+x_2), \quad \text{if}\  (x_1, x_2) \in \Omega,\ 0<x_1 \leqslant a,
\end{aligned}
\right.
$$
where $\varepsilon \in (0;a)$ is an arbitrary small number  and $\beta(x_1)=(a-x_1)(x_1+\varepsilon)^2$. 

For any $\varepsilon \in (0;a)$ we have $v_{3,\varepsilon} \in \mathbb{K}$, and $v_{3,k} \geqslant u$ in $\Omega$. It is also evident that $\mathcal{M}^{v_{3,\varepsilon}}_0 \subset \mathcal{M}^u_0$  (see Fig.~\ref{Example-Omega} (right)); in other words, for any $\varepsilon$ the function $v_{3,\varepsilon}$ has  smaller coincidence set that $u$.

First, we set $q^*=\nabla u$, $\lambda=\left[\frac{\partial u}{\partial \mathbf{n}}\right]$. Taking into account (\ref{6.1}) and appyling the estimate (\ref{estimate1}), we obtain by direct calculations the following equalities:
\begin{align*}
\|\nabla (v_{3,\varepsilon}-u)\|_{\Omega} &=\frac{2}{15 \sqrt{7}}\varepsilon^2 \mathscr{A}(a,\varepsilon),\\
\mathfrak{M}(v_{3,\varepsilon}, \nabla u, \left[\frac{\partial u}{\partial \mathbf{n}}\right], 0) &=\frac{2}{15 \sqrt{7}}\varepsilon^2 \mathscr{A}(a,\varepsilon) +4\sqrt{6} \varepsilon^{11/4} \mathscr{B}(a,\varepsilon),
\end{align*}
where
\begin{align*}
\mathscr{A}(a,\varepsilon)&=(3a^{10}+30a^9\varepsilon+135a^8\varepsilon^2+360 a^7\varepsilon^3)^{1/2} +o(\varepsilon^2),\\
\mathscr{B}(a,\varepsilon)&=
\left(\frac{a^3}{35}-\frac{\varepsilon a^2}{105}-\frac{\varepsilon^2 a}{231}+ 
\frac{\varepsilon^3}{429}\right)^{1/2}.
\end{align*}
Thus, the efficiency index takes the form
\begin{equation} \label{6.2}
1 \leqslant \frac{\mathfrak{M}(v_{3,\varepsilon}, \nabla u, \left[\frac{\partial u}{\partial \mathbf{n}}\right], 0)}
{\|\nabla (v_{3,\varepsilon}-u)\|_{\Omega} }= 1+30 \sqrt{42} \varepsilon^{3/4} \frac{\mathscr{B}(a,\varepsilon)}{\mathscr{A}(a,\varepsilon)}.
\end{equation}
Obviosly, the last term on the right-hand side of (\ref{6.2}) tends to zero as $\varepsilon \to 0$.

Next, we take $q^*=\nabla v_{3,\varepsilon}$ and $\lambda=\left[\frac{\partial v_{3,\varepsilon}}{\partial \mathbf{n}}\right]$. Due to (\ref{6.1}), (\ref{6.3}), and the equality $\big[\frac{\partial v_{3,\varepsilon}}{\partial \mathbf{n}}\big]=\big[\frac{\partial u}{\partial \mathbf{n}}\big]$, we get 
\begin{multline*}
\mathfrak{M}\left(v_{3,\varepsilon}, \nabla v_{3,\varepsilon}, \left[\frac{\partial u}{\partial \mathbf{n}}\right],0\right)=\sqrt{6}\varepsilon
\left(\int\limits_{-\varepsilon}^0 \sqrt{-x_1} \beta(x_1) (a+x_1)^2dx_1\right)^{1/2}\\
+\frac{a}{\pi} \left(\|\Delta v_{3,\varepsilon}\|^2_{\Omega_-}+\|\Delta v_{3,\varepsilon}\|^2_{\Omega_+}\right)^{1/2}=4\sqrt{6}\varepsilon^{11/4}\mathscr{B}(a,\varepsilon)+\frac{2}{3\pi} \sqrt{\frac{2}{35}}a \varepsilon^2 \mathscr{C}(a,\varepsilon),
\end{multline*}
where 
$\mathscr{C}(a,\varepsilon)=(37 a^8+296 a^7\varepsilon +2716 a^6 \varepsilon^2 -1288  a^5\varepsilon^3)^{1/2}+ o(\varepsilon^2)$.

In this case, the majorant (\ref{estimate1}) has the efficiency index
$$
1 \leqslant \frac{\mathfrak{M}(v_{3,\varepsilon}, \nabla v_{3,\varepsilon}, \left[\frac{\partial v_{3,\varepsilon}}{\partial \mathbf{n}}\right], 0)}
{\|\nabla (v_{3,\varepsilon}-u)\|_{\Omega} }=
 30 \sqrt{42}\varepsilon^{3/4} \frac{\mathscr{B}(a,\varepsilon)}{\mathscr{A}(a,\varepsilon)} +\frac{a\sqrt{10}}{\pi} \frac{\mathscr{C}(a,\varepsilon)}{\mathscr{A}(a,\varepsilon)}.
$$
It is easy to see that if $\varepsilon$ tends to zero  then the efficiency index can not exceed $3.54$.
\end{exa}

\medskip

\begin{rem}
In the above examples,
rather simple functions $q^*$ and $\lambda$ (constructed directly by means of the function
$v$)  provide quite realistic bounds of the error. In more
complicated examples, such a simple choice might lead to a considerable
overestimation of the error. In this case, so defined $q$ and $\lambda$
may be considered as a starting point for the iteration process of majorant
minimization that generates a monotonically decreasing sequence
of numbers, which are guaranteed upper bounds of the error.
\end{rem}

\subsection*{Acknowledgments}
D.A. was partly supported by the ''RUDN University Program 5-100'' and by the Russian Foundation of Basic Research (RFBR) through the grant 17-01-00678. 

\noindent
S.R. was partly supported by the Russian Foundation of Basic Research (RFBR) through the grant 17-01-00099.

\bibliography{Bibliography(ThinObstacle)}
\end{document}